\documentclass[10pt, leqno]{article}
\usepackage{amsmath,amsfonts,amsthm,amscd,amssymb}
\usepackage{pb-diagram}
\numberwithin{equation}{section}
\theoremstyle{plain}
\newtheorem{theo}{Theorem}[section]
\newtheorem{prop}[theo]{Proposition}
\newtheorem{coro}[theo]{Corollary} 
\newtheorem{lemm}[theo]{Lemma}
\theoremstyle{definition}
\newtheorem{defi}[theo]{Definition}
\newtheorem{rema}[theo]{Remark}
\newtheorem{theo-defi}[theo]{Theorem-Definition}
\newtheorem{prop-defi}[theo]{Proposition-Definition}
\newtheorem{rema-defi}[theo]{Remark-Definition}
\newtheorem{exem-defi} [theo]{Example-Definiton}

\newtheorem{exem}[theo]{Example}

\newtheorem{prob}[theo]{Problem}

\def \bet{\beta}
\def \bul{\bullet}
\def \col{\colon}

\def \del{\delta}

\def \Gam{\Gamma}

\def \kap{\kappa}

\def \Lo{\Longrightarrow}
\def \lo{\longrightarrow}
\def \lom{\longmapsto}
\def \mab{\mathbb}

\def \Om{\Omega}
\def \om{\omega}
\def \ol{\overline}
\def \os{\overset}
\def \parno{\par\noindent}

\def \sig{\sigma}

\def \sus{\subset}

\def \ul{\underline}
\def \us{\underset}

\def \vil{\varinjlim}
\def \vpl{\varprojlim}

\def \wh{\widehat}
\def \wt{\widetilde}
\bigskip

\newcommand{\getsfrom}
{\ensuremath{\longleftarrow\kern-.
52em\lower-.1ex\hbox%
{$\shortmid\,$}}}

\begin{document}

\title{Artin-Mazur heights and Yobuko heights of 
proper log smooth schemes of Cartier type, and Hodge-Witt decompositions 
and Chow groups of quasi-$F$-split threefolds} 
\author{Yukiyoshi Nakkajima
\date{}\thanks{ 
The author is supported from two JSPS
Grant-in-Aid's for Scientific Research (C)
(Grant No.~80287440, 18K03224).\endgraf}}
\maketitle

$${\bf Abstract}$$
Let $X/s$ be a proper log smooth scheme of Cartier type over 
a fine log scheme whose underlying scheme is 
the spectrum of a perfect field $\kap$ of characteristic $p>0$. 
In this article we prove that the cohomology of ${\cal W}({\cal O}_X)$ 
is a finitely generated ${\cal W}(\kap)$-module 
if the Yobuko height of $X$ is finite. 
As an application, we prove that the crystalline cohomology of 
a proper smooth threefold $Y$ 
over $\kap$ has the Hodge-Witt decomposition 
if the Yobuko height of $Y$ is finite and we prove that 
the $p$-primary torsion part of 
the Chow group of codimension $2$ of $Y$ is of finite cotype. 
These are nontrivial generalizations of results in \cite{jr} and \cite{j}. 
We also prove a fundamental inequality between 
the Artin-Mazur heights and the Yobuko height of $X/s$  
if $X/s$ satisfies natural conditions.

\section{Introduction}\label{sec:int} 
Let $\kap$ be a perfect field of characteristic $p>0$.  
Let $\ol{\kap}$ be an algebraic closure of $\kap$. 
Let $\sig \col \ol{\kap}\lo \ol{\kap}$ be 
the Frobenius automorphism of $\ol{\kap}$. 
Let ${\cal W}$ (resp.~${\cal W}_n$ $(n\in {\mab Z}_{>0})$) 
be the Witt ring of $\kap$  
(resp.~the Witt ring of $\kap$ of length $n$).  
Let $Z$ be a proper scheme over $\kap$  
and let $q$ be a nonnegative integer. 
Let $\Phi^q_{Z/\kap}$ 
be the Artin-Mazur group functor of $Z/\kap$ in degree $q$, 
that is, $\Phi^q_{Z/\kap}$ is the following functor$:$ 
$$\Phi^q_{Z/\kap}(A):={\rm Ker}(H^q_{\rm et}
(Z\otimes_{\kap}A,{\mab G}_m)
\lo H^q_{\rm et}(Z,{\mab G}_m)) \in ({\rm Ab})$$
for artinian local $\kap$-algebras $A$'s with residue fields $\kap$ (\cite{am}).  
If  $\Phi^q_{Z/\kap}$ is pro-representable by a formal group scheme over $\kap$, 
then we denote the height of $\Phi^q_{Z/\kap}$ by $h^q(Z/\kap)$. 
We call $h^q(Z/\kap)$ the $q$-th Artin-Mazur height of $Z/\kap$. 
\par 
Let $Y$ be a proper smooth scheme over $\kap$. 
In \cite{idw} Illusie has constructed the following slope spectral sequence 
\begin{align*} 
E_1^{ij}=H^j(Y,{\cal W}\Om^i_Y)\Lo 
H^{i+j}_{\rm crys}(Y/W) 
\end{align*} 
by generalizing the slope spectral sequence 
constructed by Bloch in \cite{bl}. 
It is well-known that the $E_1$-term 
$H^j(Y,{\cal W}\Om^i_{Y})$ is not a finitely generated 
${\cal W}$-module in general. 
For example, if the Artin-Mazur group functor of $Y/\kap$ in degree $j$ is 
pro-represented by a 1-dimensional formal Lie group with infinite height, 
then $H^j(Y,{\cal W}({\cal O}_X))\otimes_{\cal W}{\cal W}(\ol{\kap})
\simeq \ol{\kap}_{\sig}[[V]]$, 
where $F=0$ and $aV=Va^p$ $(a\in \ol{\kap})$ on 
$\ol{\kap}_{\sig}[[V]]$ ($\ol{\kap}_{\sig}$ means the last equality). 
\par 
In this article we are interested in the finitely generatedness of 
$H^q(Y,{\cal W}({\cal O}_Y))$ $(q\in {\mab N})$ and 
its remarkable consequences for threefolds over $\kap$. 
Let $F\col Y\lo Y$ be the Frobenius endomorphism of $Y$. 
In \cite{mr} Mehta and Ramanathan 
have given the definition of the $F$-splitness of $Y$:   
$Y$ is said to be $F$-split if 
the pull-back 
$F^*\col {\cal O}_Y\lo F_{*}({\cal O}_Y)$ 
has a section of ${\cal O}_Y$-modules.
In \cite{j} and \cite{jr} 
Joshi and Rajan have proved the following theorem: 

\begin{theo}[{\bf \cite{j}, \cite{jr}}]\label{theo:fiw}
Assume that $Y$ is $F$-split. 
Then 
$H^q(Y,{\cal W}({\cal O}_Y))$ $(q\in {\mab N})$ is 
a finitely generated ${\cal W}$-module. 
\end{theo}
\parno 
As far as we know, no nontrivial generalization of this theorem 
has been known. To generalize this Joshi-Rajan's theorem, we recall 
the definition of Yobuko height introduced in \cite{y}. 
\par 
Let $F\col {\cal W}_n(Y)\lo {\cal W}_n(Y)$ be  
the Frobenius endomorphism of ${\cal W}_n(Y)$. 
Let 
$F^*\col {\cal W}_n({\cal O}_Y)\lo F_{*}({\cal W}_n({\cal O}_Y))$ 
be the pull-back of $F$. 
Recently Yobuko has generalized the notion of the $F$-splitness: 
he has introduced the notion of the quasi-$F$-split height 
$h_F(Y)$ of $Y$ 
in a highly nontrivial way (\cite{y}) as follows. 
(In [loc.~cit.] he has denoted it by ${\rm ht}^S(Y)$.) 
It is the minimum of positive integers $n$'s such that 
there exists a morphism 
$\rho \col F_{*}({\cal W}_n({\cal O}_Y))\lo {\cal O}_Y$ 
of ${\cal W}_n({\cal O}_Y)$-modules 
such that 
$\rho \circ F^*\col {\cal W}_n({\cal O}_Y)\lo {\cal O}_Y$ 
is the natural projection. 
(If there does not exist such $n$, then  we set $h_F(Y)=\infty$.)  
In this article we call the quasi-$F$-split height the {\it Yobuko height}. 
It seems to us that the Yobuko height $h_F(Y)$ is a mysterious invariant of $Y$. 
It plays a central role in this article. 
Following \cite{y}, we say that $Y$ is {\it quasi-$F$-split} 
if $h_F(Y)<\infty$. 
\par
Next let us recall what has been known about Yobuko heights. 
\par 
In \cite{y} Yobuko has proved an equality 
$h_F(Y)=h^d(Y/\kap)$ for a Calabi-Yau variety $Y$ over $\kap$ 
of any dimension $d$.  
\par 
Let $s$ be the log point of $\kap$.  
Let $X/s$ be a proper simple normal crossing log scheme of 
pure dimension $d$. 
(In this article we do not recall fundamental notions of log geometry 
in \cite{klog1}, \cite{klog2}, \cite{hk} and \cite{nb}.) 
If the following three conditions 
\par 
(1) $H^{d-1}(X,{\cal O}_X)=0$ if $d\geq 2$, 
\par 
(2) $H^{d-2}(X,{\cal O}_X)=0$ if $d\geq 3$, 
\par 
(3) $\Om^d_{X/s}\simeq {\cal O}_X$
\parno
hold, then we have proved an equality 
$h_F(\os{\circ}{X})=h^d(\os{\circ}{X}/\kap)$ in \cite{ny}. 
Here $\Om^d_{X/s}$ is the $d$-th wedge product of sheaves 
of logarithmic differential forms on $X/s$ and 
$\os{\circ}{X}$ is the underlying scheme of $X$. 
Yobuko has also proved that 
$h^2(Z/\kap)= h_F(Z)$ for an abelian surface $Z/\kap$ (unpublished). 
(In particular, in these cases $X$ (resp.~$Z$) is $F$-split if and only if 
$h^d(\os{\circ}{X}/\kap)=1$ (resp.~$h^2(Z/\kap)=1$). This is nontrivial.) 
In \cite{yoh} he has given an example 
such that $h^2(Y/\kap)< h_F(Y)$ for an Enriques surface $Y$ over $\kap$: 

\smallskip 
\begin{equation*}
\begin{tabular}{|l|l|l|} 
\hline 
$Y$  & $h^2(Y/\kap)$  &$h_F(Y)$\\ 
\hline    
Enriques surface 
when $p>2$ & 0 & $h_F(\wt{Y})$\\ 
\hline    
classical Enriques surface 
when $p=2$ & 0 & $\infty$\\ 
\hline
singular Enriques surface when $p=2$  & 0 & $1$\\ 
\hline 
supersingular Enriques surface when $p=2$  & 0 & $\infty$\\ 
\hline 
\end{tabular}
\label{tab:intop}
\end{equation*}
Here $\wt{Y}$ is the K3 cover of $Y$.

\par 
Now we state the main results in this article. 
\par 
Let $s$ be a fine log scheme whose underlying scheme is 
${\rm Spec}(\kap)$ ($s$ is not necessarily the log point of $\kap$). 
Let ${\cal W}(s)$ be the canonical lift of $s$ over ${\rm Spf}({\cal W})$. 
Let $X/s$ be a proper log smooth scheme of Cartier type. 
Let ${\cal W}\Om^{\bul}_X$ be the log de Rham-Witt complex of $X/s$ 
and let $H^q_{\rm crys}(X/{\cal W}(s))$ $(q\in {\mab N})$ 
be the log crystalline cohomology of $X/{\cal W}(s)$. 
Following \cite{ir} in the trivial logarithmic case, we say that 
$X/s$ is of {\it log Hodge-Witt type} if 
$H^j(X,{\cal W}\Om^i_X)$ is a finitely generated ${\cal W}$-module 
for any $i,j\in {\mab N}$. 
(We do not use a phrase: ``$X/s$ is log Hodge-Witt''.) 
If $X/s$ is of log Hodge-Witt type, 
then the slope spectral sequence 
\begin{align*} 
E_1^{ij}=H^j(X,{\cal W}\Om^i_X)\Lo H^{i+j}_{\rm crys}(X/{\cal W}(s))
\end{align*} 
of $X/s$ degenerates at $E_1$ and 
there exists the following log Hodge-Witt decomposition 
for the log crystalline cohomology of $X/s$: 
\begin{align*} 
H^q_{\rm crys}(X/{\cal W}(s))=\bigoplus_{i+j=q}H^j(X,{\cal W}\Om^i_X) 
\quad (q\in {\mab N})
\end{align*}
by the log version of Illusie-Raynaud's theorem in \cite{ir} 
(cf.~\cite{lodw}). 
Let $\os{\circ}{X}$ be the underlying scheme of $X$. 
If $\dim \os{\circ}{X}=1$, then $X/s$ is of log Hodge-Witt type.

\par 
The key theorem in this article is the following:  

\begin{theo}\label{theo:dw}
Assume that $\os{\circ}{X}$ is quasi-$F$-split. 
Then $H^q(X,{\cal W}({\cal O}_X))$ $(q\in {\mab N})$ is a 
finitely generated ${\cal W}$-module. 
\end{theo}

\parno 
This is a highly nontrivial generalization of (\ref{theo:fiw}). 
To prove this theorem, we prove the following:  

\begin{theo}\label{theo:fs}
Assume that $\os{\circ}{X}$ is quasi-$F$-split. 
Then the dimensions $\dim_{\kap}H^q(X,B_n\Om^1_{X/s})$'s for 
all $q$'s and  all $n$'s are bounded. 
Here $B_n\Om^1_{X/s}$ $(n\in {\mab N})$ is a well-known 
sub ${\cal O}_X$-module of $F^n_*(\Om^1_{X/s})$, 
where $F\col X\lo X$ is 
the Frobenius endomorphism of $X$. 
\end{theo}

\parno 
Using the log version of the Serre exact sequence in \cite{smxco} 
(this has been proved in \cite{ny}), 
we can obtain (\ref{theo:dw}) by (\ref{theo:fs}) in a standard way 
(cf.~\cite{smxco}, \cite{ir}). 
As a corollary of (\ref{theo:fs}), 
we also obtain the following: 

\begin{coro}\label{coro:bb}
Assume that $\os{\circ}{X}$ is quasi-$F$-split. 
Let $f\col X\lo s$ be the structural morphism.  
Set $B_{\infty}\Om^1_{X/s}:=\vil B_n\Om^1_{X/s}$. 
Here we take the inductive limit as abelian sheaves on $\os{\circ}{X}$ 
and the transition morphisms are the natural inclusion morphisms. 
Consider $B_{\infty}\Om^1_{X/s}$ as a sheaf of 
$f^{-1}(\kap)$-submodules of $\Om^1_{X/s}$ 
$((\ref{coro:j})$ below$)$. 
Assume that $\os{\circ}{X}$ is quasi-$F$-split. 
Then $\dim_{\kap}H^q(X,\Om^1_{X/s}/B_{\infty}\Om^1_{X/s})$ 
$(q\in {\mab N})$ are finite. 
\end{coro} 

\parno 
This corollary implies the tangent space of the 
``pro-representable part'' of the formal completion 
of the second Chow group of a proper smooth 
surface over $\kap$ 
due to Stienstra (\cite{stjf}) 
is finite dimensional if it is quasi-$F$-split. 
See (\ref{rema:fdst}) below in the text for the more detailed explanation. 

\par 
In the course of the proof of (\ref{theo:dw}), 
we obtain the following unexpected result as a bonus:

\begin{theo}[{\bf Fundamental inequality between Artin-Mazur heights and 
a Yobuko height}]\label{theo:ht}
Let $X/s$ be a proper log smooth scheme 
of Cartier type. 
Let $q$ be a nonnegative integer. 
Assume that $H^q(X,{\cal O}_X)\simeq \kap$, 
that $H^{q+1}(X,{\cal O}_X)=0$ and 
that the Bockstein operator 
\begin{align*} 
\bet \col H^{q-1}(X,{\cal O}_X)\lo 
H^q(X,{\cal W}_{n-1}({\cal O}_X))
\end{align*} 
arising from the following exact sequence 
\begin{align*} 
0\lo {\cal W}_{n-1}({\cal O}_X)\os{V}{\lo} {\cal W}_n({\cal O}_X)
{\lo} {\cal O}_X\lo 0
\end{align*} 
is zero for any $n\in {\mab Z}_{\geq 2}$. 
Here $V\col  {\cal W}_{n-1}({\cal O}_X)\lo  {\cal W}_{n}({\cal O}_X)$ 
is the Verschiebung morphism. 
Assume that the functor $\Phi^q_{\os{\circ}{X}/\kap}$ is pro-representable. 
Then 
\begin{align*} 
h^q(\os{\circ}{X}/\kap)\leq h_F(\os{\circ}{X}).
\end{align*}  
In particular, if $h^q(\os{\circ}{X}/\kap)=\infty$, then 
$h_F(\os{\circ}{X})=\infty$. 
\end{theo} 

\par 
Before we proved this theorem, 
we had not even imagined that a relation between 
$h^q(\os{\circ}{X}/\kap)$ and $h_F(\os{\circ}{X})$ 
(even $h^{\dim \os{\circ}{X}}(\os{\circ}{X}/\kap)$ and $h_F(\os{\circ}{X})$) 
for a general $X/s$ as in (\ref{theo:ht}) exists because
the definitions of $h^q(\os{\circ}{X}/\kap)$ and $h_F(\os{\circ}{X})$ 
are completely different. 
After we have proved this theorem, we have been convinced that 
this theorem is true by the examples already stated. 
The theorem (\ref{theo:ht}) tells us that the Yobuko height of $\os{\circ}{X}$ 
is a upper bound of all Artin-Mazur heights of $\os{\circ}{X}/\kap$ 
under the assumptions in (\ref{theo:ht}). 
(\ref{theo:ht}) tells us a partial clear reason 
why (\ref{theo:dw}) holds. 
Indeed, $H^q(X,{\cal W}({\cal O}_X))$ is a free ${\cal W}$-module 
of rank $h^q(\os{\circ}{X}/\kap)$ 
if $h^q(\os{\circ}{X}/\kap)<\infty$
because $H^q(X,{\cal W}({\cal O}_X))$ is isomorphic to 
the Dieudonn\'{e} module of $\Phi^q_{\os{\circ}{X}/\kap}$ (\cite{am}), 
which is a free ${\cal W}$-module of rank $h^q(\os{\circ}{X}/\kap)$ 
(if $h^q(\os{\circ}{X}/\kap)<\infty$).

\par 
As a corollary of (\ref{theo:dw}), we also obtain the following 
by using the log version of a theorem in 
\cite{idw} (cf.~\cite{lodw}, \cite{ndw}): 

\begin{coro}\label{coro:adge1}
Assume that $\os{\circ}{X}$ is quasi-$F$-split 
and that $\dim \os{\circ}{X}=2$. 
Then $X/s$ is of log Hodge-Witt type. 
\end{coro}

\parno 
In \cite{jr} Joshi and Rajan have proved that 
a proper smooth $F$-split surface over $\kap$ is ordinary. 
Hence, by a fundamental theorem in \cite{ir},  
it is of Hodge-Witt type.  
The corollary (\ref{coro:adge1}) is a generalization of 
this result in two directions: the logarithmic case and the case 
where the Yobuko height is finite.
(A proper smooth scheme over $\kap$ 
with finite Yobuko height 
is far from being ordinary in general.)

For the 3-dimensional case, Joshi has proved the following theorem in \cite{j}: 

\begin{theo}[{\bf \cite{j}}]\label{theo:j}
Let $Y/\kap$ be a proper smooth scheme of dimension $3$. 
Then $Y/\kap$ is of Hodge-Witt type if and only if 
$H^q(Y,{\cal W}({\cal O}_Y))$ $(q\in {\mab N})$ is a 
finitely generated ${\cal W}$-module. 
\end{theo} 

\parno 
As an immediate corollary of (\ref{theo:dw}) and (\ref{theo:j}), 
we obtain one of the following main results in this article: 

\begin{coro}\label{coro:dge1}
Let $Y/\kap$ be as in {\rm (\ref{theo:j})}. 
Assume that $Y$ is quasi-$F$-split.  
Then $Y/\kap$ is of Hodge-Witt type. 
\end{coro}
If we prove the log version of Ekedahl's duality (\cite{ekd}) 
for dominoes associated to the differential: 
$d\col H^j(X,{\cal W}\Om^i_X)\lo H^j(X,{\cal W}\Om^{i+1}_X)$, 
then one can obtain the log version of (\ref{coro:dge1}) 
as in \cite[(6.1)]{j}.  
We would like to discuss this in a future paper. 
\par 

\par 
In \cite[II (4.1)]{konp} Kato has proved that the spectral sequence obtained 
by the $p$-adic nearby cycle sheaf of a proper smooth scheme with 
dimension less than $p-1$ over 
a complete discrete valuation ring of mixed characteristics 
degenerates at $E_2$ if the special fiber of this scheme 
is of Hodge-Witt type. 
Thus we obtain the following as a corollary of (\ref{coro:dge1}): 

\begin{coro}\label{coro:pnc}
Let ${\cal V}$ be a complete discrete valuation ring of 
mixed characteristics $(0,p)$ with perfect residue field $\kap$. 
Set $K:={\rm Frac}({\cal V})$. 
Let $\ol{K}$ be an algebraic closure of $K$ 
and $\ol{\cal V}$ the integer ring of $\ol{K}$. 
Let ${\cal Z}$ be a proper smooth scheme over ${\cal V}$. 
Set $Z:={\cal Z}\otimes_{\cal V}\kap$ and 
${\mathfrak Z}:={\cal Z}\otimes_{\cal V}K$. 
Set $\ol{\mathfrak Z}:={\mathfrak Z}\otimes_K\ol{K}$ 
and $\ol{Z}:=Z\otimes_{\kap}\ol{\kap}$. 
Let $\ol{\iota} \col \ol{Z} \os{\sus}{\lo} \ol{\cal Z}$ 
be the natural closed immersion 
and $\ol{j}\col \ol{\mathfrak Z}\lo  \ol{\cal Z}$ 
the natural open immersion. 
Assume that $Z$ is quasi-$F$-split and $\dim Z<p-1$. 
Then the following spectral sequence 
\begin{align*} 
E_2^{qr}=H^q(\ol{Z},\ol{i}{}^*R^r\ol{j}_*({\mab Z}/p^n))\Lo 
H^{q+r}_{\rm et}(\ol{Z},{\mab Z}/p^n) \quad (n\in {\mab Z}_{\geq 1})
\end{align*} 
degenerates at $E_2$.  
\end{coro} 

\par 
By using Ekedahl's theorem and his remark 
in \cite{ir} (cf.~\cite[II (2.5)]{konp}), 
we obtain the following as a corollary of (\ref{coro:dge1}):

\begin{coro}\label{coro:dgtfye1}
Let the notations be as in {\rm (\ref{coro:dge1})}. 
Then the following hold$:$
\par 
$(1)$ The following spectral sequence 
\begin{align*} 
E_1^{ij}=H^j(Y,{\cal W}_n\Om^i_Y)\Lo H^{i+j}_{\rm crys}(Y/{\cal W}_n)
\tag{1.10.1}\label{ali:wnoy}
\end{align*} 
degenerates at $E_2$ for all $n\in {\mab Z}_{\geq 1}$.  
\par 
$(2)$ 
If the operator $F\col H^j(Y,{\cal W}_n\Om^i_Y)\lo 
H^j(Y,{\cal W}_n\Om^i_Y)$ $(\forall i,j)$ is injective, 
especially if $H^q_{\rm crys}(Y/{\cal W})$ is torsion-free for $2\leq \forall q\leq 5$, 
then the spectral sequence {\rm (\ref{ali:wnoy})} 
degenerates at $E_1$ for all $n\in {\mab Z}_{\geq 1}$.  
\end{coro}
In \cite{y} Yobuko has proved that 
the spectral sequence (\ref{ali:wnoy}) for the case $n=1$ degenerates at 
$E_1$ for a Calabi-Yau variety of any dimension $d$ with finite $d$-th 
Artin-Mazur height if $d\leq p$ 
by proving that it has a smooth lift over ${\cal W}_2$ and using 
a famous theorem of Deligne-Illusie (\cite{di}). 
In \cite{ny} we have generalized this Yobuko's theorem  
for $X/s$ stated after (\ref{theo:fiw}). 
\par 
The corollary (\ref{coro:dge1}) also has an application for 
the $p$-primary torsion part of the Chow groups of 
codimension $2$ of threefolds over $\kap$ as follows.  
\par 
Let $Y$ be a proper smooth scheme over $\kap$. 
Set $Y_{\ol{\kap}}:=Y\times_{\kap}\ol{\kap}$. 
Let ${\cal W}_n\Om^{\bul}_{Y_{\ol{\kap}},{\rm log}}$ 
$(i\in {\mab N})$ be the complex of sheaves of logarithmic parts of 
${\cal W}_n\Om^{\bul}_{Y_{\ol{\kap}}}$ (\cite{idw}). 
Let $\ol{p}\col  {\cal W}_n\Om^i_{Y_{\ol{\kap}},{\rm log}}
\lo {\cal W}_{n+1}\Om^i_{Y_{\ol{\kap}},{\rm log}}$ be the induced morphism 
by the multiplication by $p\times \col 
{\cal W}_{n+1}\Om^i_{Y_{\ol{\kap}},{\rm log}}\lo
{\cal W}_{n+1}\Om^i_{Y_{\ol{\kap}},{\rm log}}$. 
Set $H^j(Y_{\ol{\kap}},{\mab Q}_p/{\mab Z}_p(i))
:=\us{{\ul{p}}}\vil H^{j-i}(Y_{\ol{\kap}},{\cal W}_n\Om^i_{Y_{\ol{\kap}},{\rm log}})$. 
Let ${\rm CH}^r(Y_{\ol{\kap}})\{p\}$ the $p$-primary torsion part of 
the Chow group of codimension $r$ of $Y_{\ol{\kap}}/\ol{\kap}$.  
In \cite{j} Joshi has proved the following theorem
by using (\ref{theo:j}), Ekedahl's duality for dominoes 
and the following injectivity of 
the $p$-adic Abel-Jacobi map of Gros and Suwa (\cite{gs}) 
\begin{align*} 
{\rm CH}^2(Y_{\ol{\kap}})\{p\}\os{\sus}{\lo}  
H^3(Y_{\ol{\kap}},{\mab Q}_p/{\mab Z}_p(2)) 
\end{align*} 
and their result ([loc.~cit., II (3.7)]):

\begin{theo}[{\bf \cite{j}}]\label{theo:c2}
Let $Y/\kap$ be a proper smooth scheme of pure dimension $3$. 
If $H^3(Y_{\ol{\kap}},{\cal W}({\cal O}_{Y_{\ol{\kap}}}))$ 
is a finitely generated ${\cal W}(\ol{\kap})$-module, 
then ${\rm CH}^2(Y_{\ol{\kap}})\{p\}$ 
is of finite cotype. 
\end{theo} 

\parno 
In particular, if $Y_{\ol{\kap}}/\ol{\kap}$ is a $3$-dimensional Calabi-Yau variety 
with finite third Artin-Mazur height, then ${\rm CH}^2(Y_{\ol{\kap}})\{p\}$ 
is of finite cotype. 
As far as we know, little had been known about 
${\rm CH}^2(Y_{\ol{\kap}})\{p\}$ for a $3$-dimensional proper smooth scheme $Y/\kap$ 
except Joshi's result (cf.~\cite{bm} and \cite{mr}).   
\par 
We obtain the following result as 
a corollary of (\ref{theo:j}), (\ref{coro:dge1}) and (\ref{theo:c2}): 

\begin{coro}\label{coro:c}
Let $Y/\kap$ be a proper smooth scheme of pure dimension $3$. 
Assume that $Y_{\ol{\kap}}$ is quasi-$F$-split. 
Then ${\rm CH}^2(Y_{\ol{\kap}})\{p\}$ is of finite cotype. 
\end{coro}

\par 
The contents of this article are as follows. 
\par
In \S\ref{sec:rrfr} we recall the characterization of the height of 
the Artin-Mazur formal group  of certain proper schemes over $\kap$.
This is a generalization of the characterization for 
Calabi-Yau varieties over $\kap$ 
due to Katsura and Van der Geer (\cite{vgk}) 
and this has been proved in a recent preprint \cite{ny}. 
We also recall a log version of Serre's exact sequence in \cite{smxco}, 
which has been proved in \cite{ny}.  
Using these results, we have determined the dimensions of  cohomologies 
of sheaves of closed log differential forms of degree $1$ in \cite{ny} as in \cite{vgk}. 
\par
In \S\ref{sec:rrafr} we generalize 
the log version of Serre's exact sequence to the case of higher degrees 
as in \cite{idw}.  For the generalization 
we recall theory of log de Rham-Witt complexes in \cite{lodw} and 
\cite{ndw}. In this article we use 
theory of formal de Rham-Witt complexes in \cite{ndw} 
which makes proofs of log versions of a lot of statements in \cite{idw} 
simple 
explicit calculations. 
\par
In \S\ref{sec:fdlr} we prove (\ref{theo:dw}) and (\ref{theo:fs}) by using 
the logarithmic version of a key commutative diagram in \cite{yoh}. 
We also prove (\ref{theo:ht}) by using the determination of 
the dimensions in \S\ref{sec:rrfr}. 
In this section we also prove (\ref{coro:adge1}), 
(\ref{coro:dge1}), (\ref{coro:pnc}), (\ref{coro:dgtfye1}) and (\ref{coro:c}). 
\par 
In \S\ref{sec:ham} we prove the following theorem: 

\begin{theo}\label{theo:amh}
Let $X/s$ be a proper log smooth scheme of Cartier type. 
Assume that $\Phi^q_{\os{\circ}{X}/\kap}$ is representable. 
Let $(\Phi^q_{\os{\circ}{X}/\kap})^*$ be the Cartier dual of 
$\Phi^q_{\os{\circ}{X}/\kap}$. 
Assume that $h^q(\os{\circ}{X}/\kap)$  is finite. 
Assume also that the morphism 
$F\col H^{q+1}(X,{\cal W}({\cal O}_X))
\lo H^{q+1}(X,{\cal W}({\cal O}_X))$ 
is injective. 
Then
\begin{align*} 
\dim (\Phi^q_{\os{\circ}{X}/\kap})\leq \dim_{\kap}H^{q}(X,{\cal O}_{X}), 
\tag{1.13.1}\label{ali:fadri}
\end{align*} 
\begin{align*} 
\dim ((\Phi^q_{\os{\circ}{X}/\kap})^*)\leq \dim_{\kap}H^{q-1}(X,\Om^{1}_{X/s})
\tag{1.13.2}\label{ali:faddri}
\end{align*} 
and 
\begin{align*} 
h^q(\os{\circ}{X}/\kap)\leq {\dim}_{\kap}
H^{q-1}(X,\Om^{1}_{X/s})+\dim_{\kap}H^{q}(X,{\cal O}_{X}). 
\tag{1.13.3}\label{ali:fari}
\end{align*} 
\end{theo} 
\parno
The inequality (\ref{ali:fari}) is a generalization of 
the following 
Katsura and Van der Geer's  results 
in \cite{vgkht}: 

\begin{prop}[{\bf \cite{vgkht}}]\label{prop:ine}
Let $Y/\kap$ be a Calabi-Yau variety of pure dimension $d$. 
Then $h^d(Y/\kap)\leq \dim_{\kap}H^{d-1}(Y,\Om^1_{Y/\kap})+1$. 
\end{prop} 

\parno
In the theorem (\ref{theo:amh}) 
we need not to assume almost anything: 
the degree of the Artin-Mazur formal group and 
the dimension of it are arbitrary and  
the log variety $X/s$ is very general. 
We prove this theorem by using theory of 
log de Rham-Witt complexes in \cite{lodw} and \cite{ndw}; 
the proof of (\ref{theo:amh}) is very different from that in 
\cite{vgkht}.
In the text we prove a more general inequality than (\ref{ali:fari}). 
 
\par
In \S\ref{sec:fs} we give the definition of an ordinary log scheme at a bidegree. 
In the same section we also prove that the exotic torsion of 
the log crystalline cohomology 
of an $F$-split proper log smooth scheme does not exist. 
This is a log version of Joshi's result. 
We also give concrete examples of $F$-split degenerate 
log schemes of dimension $\leq 2$.  
\par 

\par
\bigskip
\parno
{\bf Acknowledgment.} 
I would like to express my sincere gratitude to F.~Yobuko 
for sending me very attractive (for me) preprints 
\cite{y} and \cite{yoh}. 
Without his articles, I could not write this article. 

\par
\bigskip
\parno
{\bf Notation.} 
For a module $M$ and an element $f\in {\rm End}(M)$, 
${}_fM$ (resp.~$M/f$) denotes 
${\rm Ker}(f\col M\lo M)$  
(resp.~${\rm Coker}(f\col M\lo M)$). 
We use the same notation for an endomorphism of an abelian sheaf on 
a topological space.

\bigskip
\par\noindent
{\bf Convention.} 
We omit the second ``log'' in the terminology a ``log smooth (integral) log scheme''. 


\section{Results in \cite{ny}}\label{sec:rrfr} 
In this section we recall two results in \cite{ny} 
which are necessary for the proofs of (\ref{theo:dw}) and 
(\ref{theo:ht}). 
\par 
The following is a generalization of Katsura and Van der Geer's theorem
(\cite[(5.1), (5.2), (16.4)]{vgk}).  

\begin{theo}[{\bf \cite[(2.3)]{ny}}]\label{theo:nex} 
Let $\kap$ be a perfect field of characteristic $p>0$. 
Let $Z$ be a proper scheme over $\kap$.
$($We do not assume that $Z$ is smooth over $\kap$.$)$ 
Let $q$ be a nonnegative integer. 
Assume that $H^q(Z,{\cal O}_Z)\simeq \kap$,  
that $H^{q+1}(Z,{\cal O}_Z)=0$ and 
that $\Phi^q_{Z/\kap}$ is pro-representable. 
Let $V\col  {\cal W}_{n-1}({\cal O}_Z) \lo {\cal W}_n({\cal O}_Z)$ 
be the Verschiebung morphism and 
let $F\col {\cal W}_{n}({\cal O}_Z) \lo {\cal W}_n({\cal O}_Z)$ 
be the induced morphism by the Frobenius endomorphism of ${\cal W}_{n}(Z)$. 
Assume 
that the Bockstein operator 
\begin{align*} 
\bet \col H^{q-1}(Z,{\cal O}_Z)\lo 
H^q(Z,{\cal W}_{n-1}({\cal O}_Z))
\end{align*} 
arising from the following exact sequence 
\begin{align*} 
0\lo {\cal W}_{n-1}({\cal O}_Z)\os{V}{\lo} {\cal W}_n({\cal O}_Z)
{\lo} {\cal O}_Z\lo 0
\end{align*} 
is zero for any $n\in {\mab Z}_{\geq 2}$.  
Let $n^q(Z)$ be the minimum of positive integers $n$'s  
such that the induced morphism 
$$F\col H^q(Z,{\cal W}_n({\cal O}_Z))\lo H^q(Z,{\cal W}_n({\cal O}_Z))$$ 
by the $F\col {\cal W}_{n}({\cal O}_Z) \lo {\cal W}_n({\cal O}_Z)$ is not zero. 
$($If $F=0$ for all $n$, then set $n^q(Z):=\infty.)$ 
Let $h^q(Z/\kap)$ be the height of 
the Artin-Mazur formal group $\Phi^q_{Z/\kap}$ of $Z/\kap$. 
Then $h^q(Z/\kap)=n^q(Z)$. 
\end{theo}
\begin{proof} 
See \cite[(2.3)]{ny} (cf.~\cite[(5.1)]{vgk}) for the proofs of (\ref{theo:nex}).  
\end{proof}

As a corollary of (\ref{theo:nex}), we obtain the following:  

\begin{coro}[{\bf cf.~\cite[(5.6)]{vgk}, \cite[(2.4)]{ny}}]\label{coro:dim}
%
Let the assumptions be as in {\rm (\ref{theo:nex})}. 
Then the following equalities hold$:$ 
\begin{align*} 
{\rm dim}_{\kap}({}_FH^q(Z,{\cal W}_n({\cal O}_Z)))=
{\rm min}\{n, h^q(Z/\kap)-1\}, 
\tag{2.2.1}\label{ali:fdwy}
\end{align*} 
\begin{align*} 
{\rm dim}_{\kap}(H^q(Z,{\cal W}_n({\cal O}_Z))/F)
={\rm min}\{n, h^q(Z/\kap)-1\}. 
\tag{2.2.2}\label{ali:fdwhy}
\end{align*} 
\end{coro}
\begin{coro}
See \cite[(5.6)]{vgk} and \cite[(2.4)]{ny} for the proof of (\ref{coro:dim}). 
\end{coro}

Next we recall the log version of Serre's exact sequence 
in \cite{smxco}, which has been proved in \cite{ny}. 
\par 
Let $Z$ be a scheme over $\kap$. 
Let $F\col Z\lo Z$ be the absolute Frobenius endomorphism of $Z$. 
In \cite[\S7 (18)]{smxco} Serre has defined 
the following morphism of abelian sheaves
\begin{align*} 
d_n \col F_*({\cal W}_n({\cal O}_Z))
\lo F_*^n(\Om^1_{Z/\kap}) 
\end{align*} 
defined by the following formula$:$
\begin{align*} 
d_n((a_0,\ldots,a_{n-1}))
=\sum_{i=0}^{n-1}a^{p^{n-1-i}-1}_ida_i 
\quad (a_i\in {\cal O}_Z).  
\tag{2.2.3}\label{ali:sed}
\end{align*} 
(In [loc.~cit.] he has denoted $d_n$ by $D_n$.)
He has remarked that the following formula holds:
\begin{align*} 
d_n((a_0,\ldots,a_{n-1})(b_0,\ldots,b_{n-1}))=
b_0^{p^{n-1}}d_n((a_0,\ldots,a_{n-1}))+
a_0^{p^{n-1}}d_n((b_0,\ldots,b_{n-1})). 
\tag{2.2.4}\label{ali:dna} 
\end{align*} 
Hence the morphism 
$d_n\col F_*({\cal W}_n({\cal O}_Z))
\lo F_*^n(\Om^1_{Z/\kap})$ 
is a morphism of ${\cal W}_n({\cal O}_Z)$-modules. 
(This remark was not given in [loc.~cit.]. See also 
(\ref{rema:omst}) below.)

\par 
Let $s$ be as in the Introduction. 
Let $F_s\col s\lo s$ be the Frobenius endomorphism. 
Let $Y/s$ be a log smooth scheme of Cartier type. 
Set $Y':=Y\times_{s,F_s}s$. 
Let $F:=F_{Y/s}\col Y\lo Y'$ be the relative Frobenius morphism over $s$. 
The log inverse Cartier isomorphism due to Kato 
is the following isomorphism of sheaves of ${\cal O}_{Y'}$-modules 
(\cite[(4.12) (1)]{klog1}): 
\begin{align*} 
C^{-1}\col \Om^i_{Y'/s}
\os{\sim}{\lo} F_{*}({\cal H}^i(\Om^{\bul}_{Y/s})). 
\tag{2.2.5}\label{ali:oxs}
\end{align*} 
Because $\os{\circ}{F}$ 
is a homeomorphism (\cite[XV Proposition 2 a)]{sga5-2}), 
we can also express (\ref{ali:oxs}) as the equality 
\begin{align*} 
C^{-1}\col \Om^i_{Y'/s}
= {\cal H}^i(F_{*}(\Om^{\bul}_{Y/s})) 
\tag{2.2.6}\label{ali:oxfws}
\end{align*} 
of ${\cal O}_{Y'}$-modules. 
Set $Y^{\{n\}}:=Y^{\{n-1\}}\times_{s,F_s}s$ $(n\in {\mab Z}_{\geq 2})$ 
and 
$F^n:=F_{Y^{\{n\}}/s} 
\circ \cdots  \circ F_{Y'/s} \circ F_{Y/s}\col Y\lo Y^{\{n\}}$. 
Set $B_0\Om^i_{Y/s}:=0$, 
$B_1\Om^i_{Y/s}:=F_{*}
({\rm Im}(d\col {\cal O}_Y\lo \Om^1_{Y/s}))$  
and 
$Z_0\Om^i_{Y/s}:=\Om^i_{Y/s}$. 
Then $B_1\Om^i_{Y/s}$ (resp.~$Z_0\Om^i_{Y/s}$)
is a sheaf of $F_*({\cal O}_Y)$-module 
(resp.~${\cal O}_Y$-module). 
We define sheaves of $F^n_*({\cal O}_Y)$-modules 
$B_n\Om^i_{Y/s}$ and $Z_n\Om^i_{Y/s}$ on $(Y^{\{n\}})_{\rm zar}$ 
(not on $Y_{\rm zar}$)
inductively by the following equalities $(n\geq 1)$: 
$$C^{-1}\col B_{n-1}\Om^i_{Y'/s}=
B_n\Om^i_{Y/s}/F^{n-1}_*(B_1\Om^i_{Y/s}), \quad 
C^{-1}\col Z_{n-1}\Om^i_{Y'/s} =
Z_n\Om^i_{Y/s}/F^{n-1}_*(B_1\Om^i_{Y/s}).$$ 
(Because $Y'/s$ is log smooth and of Cartier type, 
these definitions are well-defined.)
Then we have the following inclusions of 
sheaves of $F^{n}_*({\cal O}_Y)$-modules 
(not only the inclusion of abelian sheaves): 
\begin{align*}
0&\subset F^{n}_*(B_1\Om^i_{Y/s})\subset \cdots \subset 
F_*(B_n\Om^i_{Y/s})\subset  
B_{n+1}\Om^i_{Y/s}
\\ 
&\subset Z_{n+1}\Om^i_{Y/s} \subset F_*(Z_n\Om^i_{Y/s}) 
\subset  \cdots 
\subset  F^{n}_*(Z_1\Om^i_{Y/s}) \subset F^{n+1}_*(\Om^i_{Y/s}).
\end{align*}
Because the projection $Y'\lo Y$ induces an isomorphism 
$\os{\circ}{Y}{}'\os{\sim}{\lo} \os{\circ}{Y}$, 
we have the following 
the composite isomorphism of the projections 
\begin{align*} 
(Y^{\{n\}})^{\circ}\os{\sim}{\lo} 
(Y^{\{n-1\}})^{\circ}\os{\sim}{\lo} \cdots \os{\sim}{\lo}
\os{\circ}{Y}{}'\os{\sim}{\lo} \os{\circ}{Y}
\end{align*} 
over 
\begin{align*} 
\os{\circ}{s}\os{\sim}{\lo} 
\os{\circ}{s}\os{\sim}{\lo} \cdots \os{\sim}{\lo}
\os{\circ}{s}\os{\sim}{\lo} \os{\circ}{s}. 
\end{align*}
Hence we can consider 
$B_n\Om^i_{Y/s}$ and $Z_n\Om^i_{Y/s}$ 
as $F^n_*({\cal O}_Y)$-modules, where 
$F\col Y\lo Y$ is the absolute Frobenius endomorphism of $Y$. 
(I prefer the ways of the definitions of  $Z_n\Om^i_{Y/s}$ and 
$B_n\Om^i_{Y/s}$ above to the ways of
Illusie's definitions of them in \cite[0 (2.2.2)]{idw} 
in the trivial logarithmic case because in our definition, 
it is not necessary to consider $Z_n\Om^i_{Y/s}$ and 
$B_n\Om^i_{Y/s}$ as abelian sheaves on $Y$ and $F^n_*({\cal O}_Y)$-modules 
separately.)

\par 
The following (\ref{prop:wus}) 
is a log version of a generalization of Serre's result in \cite{smxco}.

\begin{prop}[{\bf \cite[(3.5), (3.6)]{ny}}]\label{prop:wus}
Let $n$ be a positive integer. 
Denote the following composite morphism 
\begin{align*} 
F_*({\cal W}_n({\cal O}_Y))\os{d_n}{\lo} 
F^n_*(\Om^1_{\os{\circ}{Y}/\kap})\lo F^n_*(\Om^1_{Y/s})
\end{align*} 
by $d_n$ again. Then the following fold$:$
\par 
$(1)$ 
The morphism $d_n$ factors through $B_n\Om^1_{Y/s}$ 
and the following sequence 
\begin{align*} 
0\lo {\cal W}_n({\cal O}_Y)\os{F}{\lo} 
F_*({\cal W}_n({\cal O}_Y))
\os{d_n}{\lo} B_n\Om^1_{Y/s}\lo 0
\tag{2.3.1;$n$}\label{ali:wnox} 
\end{align*} 
is exact. 
Here we denote the morphism 
${\cal W}_n(F^*)={\cal W}_n(F^*_{Y/s})$ $($resp.~
$F_*({\cal W}_n({\cal O}_{Y}))\lo B_n\Om^1_{Y/s})$ 
by $F$ $($resp.~$d_n)$ again by abuse of notation. 
Consequently $d_n$ induces the following isomorphism of 
${\cal W}_n({\cal O}_Y)$-modules$:$ 
\begin{align*} 
F_*({\cal W}_n({\cal O}_Y))/{\cal W}_n({\cal O}_{Y})
\os{\sim}{\lo}
B_n\Om^1_{Y/s}. 
\tag{2.3.2}\label{ali:fwo} 
\end{align*}  
\par 
$(2)$ Let $R\col {\cal W}_n({\cal O}_Y)\lo {\cal W}_{n-1}({\cal O}_Y)$ 
be the projection. 
Let $C\col B_n\Om^1_{Y/s}\lo B_{n-1}\Om^1_{Y/s}$ 
be the following composite morphism
\begin{align*} 
B_n\Om^1_{Y/s}\os{\rm proj.}{\lo} 
B_n\Om^1_{Y/s}/F^{n-1}_*(B_1\Om^1_{Y/s})
\os{C^{-1},\sim}{\longleftarrow} 
B_{n-1}\Om^1_{Y'/s}\os{\sim}{\longleftarrow}  B_{n-1}\Om^1_{Y/s}. 
\end{align*} 
Then the following diagram 
\begin{equation*} 
\begin{CD} 
F_*({\cal W}_n({\cal O}_Y))
@>{d_n}>>B_n\Om^1_{Y/s}\\
@V{F_*(R)}VV @VV{C}V \\
F_*({\cal W}_{n-1}({\cal O}_Y))
@>{d_{n-1}}>>B_{n-1}\Om^1_{Y/s}
\end{CD}
\tag{2.3.3}\label{cd:dowm}
\end{equation*} 
is commutative. 
\end{prop}
\begin{proof} 
See \cite[(3.5)]{ny} for the proof of (\ref{prop:wus}). 
\end{proof}

\begin{defi}
We call the exact sequence (\ref{ali:wnox}) of 
${\cal W}_n({\cal O}_Y)$-modules 
the {\it log Serre exact sequence of $Y/s$ of level $n$}. 
\end{defi} 

\begin{coro}[{\bf \cite[(3.7)]{ny}}]\label{coro:chob} 
Let the assumptions be as in {\rm (\ref{theo:nex})} for $Z:=\os{\circ}{Y}$. 
Then 
$H^q(Y,{\cal W}_n({\cal O}_Y))/F=H^q(Y,B_n\Om^1_{Y/s})$. 
Consequently 
\begin{equation*} 
{\rm dim}_{\kap}H^q(Y,B_n\Om^1_{Y/s})=
{\rm min}\{n, h^q(\os{\circ}{Y}/\kap)-1\}. 
\tag{2.5.1}\label{eqn:pkdefpw}
\end{equation*}   
\end{coro} 
\begin{proof} 
See \cite[(3.7)]{ny} for the proof of (\ref{coro:chob}); 
it is easy to derive this from (\ref{coro:dim}) and (\ref{prop:wus}).  
\end{proof} 

\section{A generalization of the log Serre exact sequence}\label{sec:rrafr}  
In this section we recall theory of 
formal de Rham-Witt complexes in \cite{ndw} with a slightly 
different formulation from that in [loc.~cit.] 
and we generalize the log version of Serre's exact sequence (\ref{ali:wnox}) 
by using this theory. 
That is, we prove that the following sequence 
\begin{align*} 
{\cal W}_{n+1}\Om^i_{Y/s}
\os{F}{\lo} F_{{\cal W}_n(Y)*}({\cal W}_n\Om^i_{Y/s})
\os{F_{{\cal W}_n(Y)*}(F^{n-1}d)}{\lo} 
B_n\Om^{i+1}_{Y/s}\lo 0 \quad (i\in {\mab N})
\tag{3.0.1}\label{ali:exwys}
\end{align*}
is exact in the category of ${\cal W}_{n+1}({\cal O}_Y)$-modules 
for a log smooth scheme $Y/s$ of Cartier type. 
This generalization is a log version of Illusie's generalization of 
Serre's exact sequence in \cite{idw}, 
though he has considered the exactness in [loc.~cit.] 
only in the category of abelian sheaves not in 
the category of modules of the Witt sheaves of structure sheaves. 
If the reader wants to know only the proofs of the results 
(\ref{theo:dw}) and (\ref{theo:ht}) in the 
Introduction, he can skip this section. 
However we shall use (\ref{eqn:rfvn}) below 
for the proof of (\ref{ali:fari}) and  use 
several results in this section to obtain results in the book \cite{nb}. 
\par 
Let $\kap$ be a perfect field of characteristic $p>0$. 
Let ${\cal W}$ be the Witt ring of $\kap$.  
Let $\sig$ be the Frobenius automorphism of ${\cal W}$. 

Let $({\cal T}, {\cal W})$  be a ringed topos: 
${\cal T}$ is a topos and ${\cal W}$ is the constant sheaf 
in ${\cal T}$ defined by the Witt ring ${\cal W}$.  
Let $\Om^{\bul}$ be 
a bounded complex of sheaves of torsion-free ${\cal W}$-modules 
and let $\phi \col  \Om^{{\bul}} \lo \sig_*(\Om^{\bul})$ 
be a morphism of complexes of ${\cal W}$-modules. 
Let $p$ be a 
prime number. Set 
$\Om^{{\bul}}_1:=\Om^{{\bul}}/p\Om^{{\bul}}$.
We assume that the following conditions 
$(3.0.2) \sim (3.0.6)$ hold:
\medskip
\parno
(3.0.2) $\Om^i=0$ for $i<0$.
\medskip
\parno
(3.0.3) $\Om^i$ $(\forall i \in {\mab N})$ is a sheaf of 
$p$-torsion-free and $p$-adically complete ${\cal W}$-module.
\medskip
\parno
(3.0.4) $\phi(\Om^i) \subset 
\sig_*\{\om \in p^i \Om^i~ \vert~ d\om \in p^{i+1} 
\Om^{i+1}\}$ $(\forall i\in {\mab N})$.
\medskip
\parno
(3.0.5) There exists an isomorphism of sheaves of $\kap$-modules 
$$C^{-1} \col \Om^{i}_1\os{\sim}{\lo} 
\sig_*{\cal H}^i(\Om^{\bul}_1) \quad 
(\forall i \in {\mab N}).$$ 
Here we denote by $\sig$ the Frobenius automorphism of $\kap$ 
by abuse of notation. 
\medskip
\parno
(3.0.6) A composite morphism 
(${\rm mod}~p)\circ p^{-i}\phi 
\col \Om^{i} \lo \sig_*(\Om^i) \lo \sig_*(\Om^i_1)$ factors 
through $\sig_*{\rm Ker}(d \col\Om^i_1 \lo \Om^{i+1}_{1})$ 
and the following diagram is commutative:
\begin{equation*}
\begin{CD}
\Om^{i} @>{\mod p}>> 
\Om^{i}_1\\ 
@V{p^{-i}\phi}VV  @VV{C^{-1}}V \\
\sig_*(\Om^i) @>{\mod p}>> 
\sig_*{\cal H}^i(\Om^{\bul}_1).
\end{CD}
\end{equation*}

First we recall the following: 

\begin{prop}[{\bf \cite[(6.4)]{ndw}}]\label{prop:boncom}
Let $i$ $($resp.~$n)$ be a 
non-negative $($resp.~positive$)$ integer. 
Set 
\begin{equation*}
Z^i_n:=\{\om \in \Om^i \vert~ 
d\om \in p^n\Om^{i+1}\}, \quad
B^i_n:=p^n\Om^i+ d\Om^{i-1} \quad {\rm and} \quad 
{\mathfrak W}_n\Om^i:= \sig^n_*(Z^i_n/B^i_n). 
\end{equation*}
Then the morphism $\phi \col \Om^{{\bul}} \lo \sig_*(\Om^{\bul})$ 
induces the following isomorphism of sheaves of ${\cal W}$-modules$\,:$  
\begin{equation*}
{\mathfrak W}_n\Om^i
\os{\us{\sim}{\phi}}{\lo}
\sig_*\{p^iZ^i_{n+1}/(p^{i+n}Z^i_1+p^{i-1}dZ^{i-1}_1)\}.
\tag{3.1.1}
\label{eqn:sexgg}
\end{equation*}
\end{prop}

\par
By (3.0.5) we have the following isomorphism 
\begin{equation*}
\sig_*{\mathfrak W}_1\Om^i= \sig_*{\cal H}^i(\Om^{\bul}_1)
\os{\os{C^{-1}}{\sim}}{\longleftarrow} \Om^i_1.  
\tag{3.1.2}\label{eqn:1wc}
\end{equation*}
\par
Recall the following morphisms 
\begin{align*}
&F \col {\mathfrak W}_{n+1}\Om^{i} \lo \sig_*{\mathfrak W}_n\Om^{i},\quad 
V\col \sig_*{\mathfrak W}_n\Om^{i} \lo {\mathfrak W}_{n+1}\Om^{i},\\
&d \col {\mathfrak W}_n\Om^{i} 
\lo {\mathfrak W}_n\Om^{i+1},\quad 
R \col {\mathfrak W}_{n+1}\Om^{i} 
\lo 
{\mathfrak W}_n\Om^{i}
\end{align*}
of sheaves of ${\cal W}$-modules in ${\cal T}$ as follows:
$F$ (resp.~$V$) is a morphism induced by 
${\rm id} \col \Om^{i} \lo \Om^{i}$ 
(resp.~$p\times{\rm id} 
\col \Om^{i} \lo \Om^{i}$); 
$d$ is a morphism induced by 
$p^{-n}d \col  Z^i_n \lo \Om^{i+1}$; 
$R$ is the following composite surjective morphism (cf.~\cite[(4.2)]{hk}):
\begin{align*}
{\mathfrak W}_{n+1}{\Om}^i &= 
\sig^{n+1}_*(Z^i_{n+1}/B^i_{n+1}) 
\os{\us{\sim}{p^i}}{\lo} 
\sig^{n+1}_*(p^iZ^i_{n+1}/p^iB^i_{n+1}) \tag{3.1.3}\label{ali:aplocpj}\\
& \os{{\rm proj}.}{\lo} 
\sig^{n+1}_*(p^iZ^i_{n+1}/(p^{i+n}Z^i_1+p^{i-1}dZ^{i-1}_1)) 
\os{\us{\sim}{\phi}}{\longleftarrow}
\sig^{n}_*(Z^i_n/B^i_n)={\mathfrak W}_{n}{\Om}^i.
\end{align*}
Then the following formulas hold:
\begin{equation*}
d^2=0,~ FdV=d,~FV=VF=p,~FR=R F,~dR=Rd,~VR=RV.
\tag{3.1.4}
\label{eqn:earl}
\end{equation*}


\par 
\begin{lemm}[{\bf \cite[(6.7)]{ndw}}]\label{lemm:wst}
Let $\star$ be a positive integer $n$ or nothing.
Set 
\begin{equation*}
{\mathfrak W}\Om^{\bul}=
\us{R}{\vpl}~{\mathfrak W}_n\Om^{\bul}. 
\end{equation*}
Then there exists a natural 
${\cal W}_{\star}$-module structure on 
${\mathfrak W}_{\star}\Om^i$ 
$($see the explanation after {\rm (\ref{theo:raycomp})} below$)$.
\end{lemm}

We have called ${\mathfrak W}\Om^{\bul}$ 
(resp.~${\mathfrak W}_n\Om^{\bul})$
the formal de Rham-Witt complex
(resp.~formal de Rham-Witt complex of length $n$) 
of $(\Om^{\bul}, \phi, C^{-1})$.

\par 
Set $Z\Om^i_1:=
{\rm Ker}(d \col \Om_1^i \lo \Om_1^{i+1})$ 
and $B\Om^i_1:= {\rm Im}(d \col \Om_1^{i-1} \lo \Om_1^i)$. 
In \cite[(6.16.2), (6.16.3)]{ndw} we have defined 
the following subsheaves 
$Z_n\Om^i_1$ and $B_n\Om^i_1$ of $\kap$-modules of $\sig^n_*(\Om^i_1)$ 
inductively for $n \in {\mab N}$:
\begin{equation*}
Z_0\Om_1^i:=\Om^i_1, \quad  
\quad Z_{n}\Om^i_1 /\sig^n_*(B\Om^i_1)  
\os{C^{-1}}{\os{\sim}{\longleftarrow}} 
Z_{n-1}\Om^i_1 \quad (n\in {\mab Z}_{\geq 1}),  
\tag{3.2.1}\label{eqn:defz}
\end{equation*}
\begin{equation*}
B_0\Om_1^i:=0, \quad 
B_n\Om^i_1 /\sig^n_*(B\Om^i_1)\os{C^{-1}}{\os{\sim}{\longleftarrow}} 
B_{n-1}\Om^i_1 \quad (n\in {\mab Z}_{\geq 1}). 
\tag{3.2.2}\label{eqn:defb}
\end{equation*} 

\begin{lemm}[{\bf \cite[(6.17)]{ndw}}]\label{lemm:zbn}
$(1)$ $Z_n\Om^i_1=\sig_*^n\{(Z^i_n+p\Om^i)/p\Om^i\}$ 
$(n\in {\mab Z}_{>0})$.
\par
$(2)$ $B_n\Om^i_1=
\sig_*^n\{(p^{-(n-1)}dZ^{i-1}_{n-1}+p\Om^i)/p\Om^i\}$ 
$(n\in {\mab Z}_{>0})$.
\end{lemm}

In the proof of \cite[(6.17)]{ndw} 
we have proved that the following morphisms  
\begin{equation*}
C^{-1}=p^{-i}\phi \col 
(Z^i_n+p\Om^i)/p\Om^i \lo  
\sig_*\{(Z^i_{n+1}+p\Om^i)/(p\Om^i+d\Om^{i-1})\}  \quad
(n\in {\mab Z}_{>0}) 
\tag{3.3.1}
\label{eqn:zcatr}
\end{equation*}
and 
\begin{equation*}
C^{-1}=p^{-i}\phi \col 
(p^{-(n-1)}dZ^{i-1}_{n-1}+p\Om^i)
/p\Om^i \lo 
\sig_*\{(p^{-n}dZ^{i-1}_n+p\Om^i)
/(p\Om^i+d\Om^{i-1})\} 
\quad 
(n\in {\mab Z}_{>0})
\tag{3.3.2}
\label{eqn:zcabd}
\end{equation*}
are isomorphisms of $\kap$-modules.  
Set $Z_n{\mathfrak W}_1\Om^i:=\sig^{n+1}_*\{Z_{n+1}\Om^i_1/B\Om^i_1\}$ and 
$B_n{\mathfrak W}_1\Om^i:=\sig^{n+1}_*\{B_{n+1}\Om^i_1/B\Om^i_1\}$. 
These are $\kap$-submodules of 
$\sig^n_*({\mathfrak W}_1\Om^i)$.

\par 
Let us consider the following composite morphisms of $\kap$-modules:
\begin{align*}
C \col 
Z_n{\mathfrak W}_1\Om^i &=\sig^{n+1}_*\{Z_{n+1}\Om^i_1/B\Om^i_1\} 
\os{{\rm proj}.}{\lo}
\sig^{n+1}_*\{Z_{n+1}\Om^i_1/B_2\Om^i_1\} \tag{3.3.3}\label{eqn:zomc}\\
&\os{(C^{-1})^{-1}}{\os{\sim}{\lo}} 
\sig^{n}_*\{Z_n\Om^i_1/B\Om^i_1\}=
Z_{n-1}{\mathfrak W}_1\Om^i  \quad (n\geq 0),  
\end{align*}

\begin{align*}
C \col B_n{\mathfrak W}_1\Om^i&= \sig^{n+1}_*\{B_{n+1}\Om^i_1/B\Om^i_1\} 
\os{{\rm proj}.}{\lo}
\sig^{n+1}_*\{B_{n+1}\Om^i_1/B_2\Om^i_1\} \tag{3.3.4}\label{eqn:bomc}\\
& \os{(C^{-1})^{-1}}{\os{\sim}{\lo}}  
\sig^{n}_*\{B_n\Om^i_1/B\Om^i_1\}=B_{n-1}{\mathfrak W}_1\Om^i \quad (n\geq 1). 
\end{align*}

\par 

The following (3) is a formal generalization of 
(\ref{ali:wnox}) (cf.~\cite[I (3.11)]{idw}). 

\begin{prop}[{\bf cf.~\cite[I (3.11)]{idw}}]\label{prop:fkervcoker}
$(1)$ The morphism 
$F^n \col {\mathfrak W}_{n+1}\Om^i\lo \sig^{n}_*{\mathfrak W}_1\Om^i$ 
induces the following isomorphism of sheaves of $\kap$-modules 
in ${\cal T}\!:$ 
\begin{align*} 
{\mathfrak W}_{n+1}\Om^i/V(\sig_*{\mathfrak W}_n\Om^i) 
\os{\sim}{\lo} 
Z_n{\mathfrak W}_1\Om^i. 
\tag{3.4.1}\label{ali:nwee}
\end{align*}
\par 
$(2)$ The following diagram is commutative$:$
\begin{equation*} 
\begin{CD} 
{\mathfrak W}_{n+1}\Om^i@>{F^n}>>Z_n{\mathfrak W}_1\Om^i\\
@V{R}VV @VV{C}V\\
{\mathfrak W}_n\Om^i@>{F^{n-1}}>>Z_{n-1}{\mathfrak W}_1\Om^i. 
\end{CD} 
\tag{3.4.2}\label{ali:nmfee}
\end{equation*} 
\par 
$(3)$ The morphism 
$F^{n-1}d \col {\mathfrak W}_n\Om^i\lo \sig^{n-1}_*{\mathfrak W}_1\Om^{i+1}$ 
induces the following isomorphism of abelian sheaves in ${\cal T}\!\!:$ 
\begin{align*} 
\sig_*{\mathfrak W}_n\Om^i/F{\mathfrak W}_{n+1}\Om^i\os{\sim}{\lo} 
B_n{\mathfrak W}_1\Om^{i+1}. 
\tag{3.4.3}\label{ali:nmee}
\end{align*}
\par 
$(4)$ The following diagram is commutative$:$
\begin{equation*} 
\begin{CD} 
\sig_*{\mathfrak W}_{n+1}\Om^i@>{F^nd}>>B_{n+1}{\mathfrak W}_1\Om^{i+1}_1\\
@V{R}VV @VV{C}V\\
\sig_*{\mathfrak W}_n\Om^i@>{F^{n-1}d}>>B_n{\mathfrak W}_1\Om^{i+1}_1. 
\end{CD} 
\tag{3.4.4}\label{cd:nmwee}
\end{equation*} 
\end{prop}
\begin{proof} 
(1): Because 
$F^n \col {\mathfrak W}_{n+1}\Om^i\lo \sig^{n}_*{\mathfrak W}_1\Om^i$ 
is a morphism of  sheaves of ${\cal W}$-modules in ${\cal T}$, 
we have only to prove that the morphism $F^n$ 
induces an isomorphism 
${\mathfrak W}_{n+1}\Om^i/V{\mathfrak W}_n\Om^i 
\os{\sim}{\lo} 
Z_n{\mathfrak W}_1\Om^i$ of abelian sheaves in ${\cal T}$. 
By (\ref{lemm:zbn}) (1) 
it suffices to prove that the following morphism 
\begin{align*} 
{\rm proj}. \col Z^i_{n+1} \lo  (Z^i_{n+1}+p\Om^i)/p\Om^i 
\tag{3.4.5}\label{ali:zin}
\end{align*} 
is surjective and its kernel is equal to $pZ^i_n$. 
The surjectivity is obvious. 
Let $\om$ be a local section of $Z^i_{n+1}$. 
Assume that $\om = p\eta$ with $\eta \in \Om^i$.  
Then $p^{n+1}\Om^i\owns d\om=pd\eta$.  
Since $\Om^i$ is torsion free, $d\eta \in p^n\Om^i$. 
Hence the kernel of the morphism (\ref{ali:zin}) is $pZ^i_n$. 
\par 
(2): The diagram (\ref{ali:nmfee}) is equal to 
\begin{equation*}
\begin{CD}
{\mathfrak W}_{n+1}\Om^i 
@>{\rm proj.}>> Z_{n+1}\Om^i_1/B_1\Om^i_1 
@>{\rm proj.}>>Z_{n+1}\Om^i_1/B_2\Om^i_1 \\ 
@V{R}VV @. @A{C^{-1}}A{\simeq}A   \\
{\mathfrak W}_n\Om^i 
@>{{\rm proj}.}>> 
Z_n\Om^i_1/B_1\Om^i_1@= 
Z_n\Om^i_1/B_1\Om^i_1. 
\end{CD}
\tag{3.4.6}\label{cd:bgpid}
\end{equation*}
In \cite[(6.18.2)]{ndw} we have already proved that this is commutative. 
\par 
(3): By (\ref{lemm:zbn}) (2) 
it suffices to prove that the following morphism 
\begin{align*} 
p^{-n}d \col Z^i_n \lo  
(p^{-n}dZ^i_n+p\Om^{i+1})/p\Om^{i+1} \quad (n\in {\mab Z}_{>0}) 
\end{align*} 
is surjective and its kernel is equal to $Z^i_{n+1}$. This is obvious. 
\par 
(4): 
It suffices to prove that the following diagram is commutative:  
\begin{equation*}
\begin{CD}
{\mathfrak W}_{n+1}\Om^i 
@>{p^{-(n+1)}d}>> {\mathfrak W}_{n+1}\Om^{i+1} 
@>{{\rm proj}.}>> Z_{n+1}\Om^{i+1}_1/B_2\Om^{i+1}_1 
\\ 
@V{R}VV @. @A{C^{-1}}A{\simeq}A  \\
{\mathfrak W}_n\Om^i 
@>{p^{-n}d}>> {\mathfrak W}_n\Om^{i+1}
@>{{\rm proj}.}>> 
Z_n\Om^{i+1}_1/B\Om^{i+1}_1. 
\end{CD}
\tag{3.4.7}\label{cd:bgbpd}
\end{equation*}
Consider sections 
$[\om] \in {\mathfrak W}_{n+1}\Om^i$ 
$(\om \in Z^i_{n+1})$ and
$[\eta] \in {\mathfrak W}_n\Om^i$ 
$(\eta \in Z^i_n)$ such that
$p^i\om-\phi(\eta)\in p^{i+n}Z^i_1
+p^{i-1}dZ^{i-1}_1$ ((\ref{eqn:sexgg})).
Then  $R([\om])=[\eta]$ by the definition of $R$. 
We also have the following equalities: 
\begin{align*} 
p^{-(n+1)}d(\om -p^{-i}\phi(\eta))&=
p^{-(n+1)}d\om-p^{-(i+1)}\phi (p^{-n}d\eta)
\tag{3.4.8}\label{cd:bgndp}
\end{align*}  
and 
\begin{align*} 
p^{-(n+1)}d(p^nZ^i_1+p^{-1}dZ^{i-1}_1)=p^{-1}dZ^i_1.
\tag{3.4.9}\label{cd:bgsdp}
\end{align*} 
By (\ref{lemm:zbn}) (2), this sheaf mod $p$ is contained in $B_2\Om^i_1$. 
Hence, by (\ref{cd:bgndp}) and (\ref{cd:bgsdp}), 
the right hand side on (\ref{cd:bgndp}) is equal to zero in 
$Z_{n+1}\Om^i_1/B_2\Om^i_1$. This implies that 
the diagram (\ref{cd:bgbpd}) is commutative. 
\end{proof}

\begin{prop}[{\bf \cite[(6.14)]{ndw} (cf.~\cite[I (3.19.2.1)]{idw}, \cite[p.~258]{lodw})}]
\label{prop:geni}
Let $n > r$ be two positive integers. 
Then the following sequence is exact$:$
\begin{equation*}
0 \lo \sig_*^r({\mathfrak W}_{n-r}\Om^{i-1}) 
/F^r{\mathfrak W}_n\Om^{i-1} 
\os{dV^r}{\lo}
{\mathfrak W}_n\Om^i
/V^r{\mathfrak W}_{n-r}\Om^i \os{R^{n-r}}{\lo} 
{\mathfrak W}_r\Om^i \lo 0. 
\tag{3.5.1;$r,n$}\label{ali:wnii}
\end{equation*}
Consequently the following sequence 
is exact$:$
\begin{equation*}
0 \lo \sig_*^r({\mathfrak W}\Om^{i-1})
/F^r{\mathfrak W}\Om^{i-1} \os{dV^r}{\lo}
{\mathfrak W}\Om^i/V^r{\mathfrak W}\Om^i \lo 
{\mathfrak W}_r\Om^i \lo 0. 
\tag{3.5.1;$r$}\label{eqn:rfvn}
\end{equation*} 
\end{prop}

\begin{theo}[{\bf \cite[(6.15)]{ndw} (cf.~\cite[I (3.31)]{idw})}]\label{theo:inffilv} 
Let $r$ be a non-negative integer. 
Let ${\rm Fil}^{\bul}$ be the canonical filtration on 
${\mathfrak W}\Om^i:$ 
${\rm Fil}^r{\mathfrak W}\Om^i:={\rm Ker}({\mathfrak W}\Om^i\lo 
{\mathfrak W}_r\Om^i)$. 
Then the following formula holds$:$
\begin{equation*}
{\rm Fil}^r{\mathfrak W}\Om^i=V^r\sig^r_*({\mathfrak W}\Om^i)
+dV^r\sig^r_*({\mathfrak W}\Om^{i-1}). 
\tag{3.6.1}\label{eqn:fvdvn}
\end{equation*}
\end{theo}

\begin{coro}[{\bf \cite[(6.6)]{ndw} 
(cf.~\cite[I (3.31)]{idw}, \cite[II (1.1.1)]{ir}, \cite[(2.16)]{lodw})}]\label{coro:rayr}
Let $R_{\infty}$ be the Cartier-Dieudonn\'{e}-Raynaud algebra over $\kap$. 
Let $n$ be a positive integer.
Set $R_n:=R_{\infty}/(V^nR_{\infty}+dV^nR_{\infty})$.
The canonical morphism 
\begin{equation*}
R_n\otimes_{R_{\infty}}{\mathfrak W}\Om^{\bul}\lo 
{\mathfrak W}_n\Om^{\bul} 
\tag{3.7.1}\label{eqn:raylpi}
\end{equation*}
is an isomorphism.
\end{coro}


\begin{prop}
[{\bf \cite[(6.23)]{ndw} (cf.~\cite[I (3.21.1.5)]{idw}, 
\cite[(1.20)]{lodw})}]\label{prop:dinvfn}
Let $n$ be a non-negative integer. 
Then $d^{-1}(p^n {\mathfrak W}\Om^{i+1})
=F^n{\mathfrak W}\Om^i$.
\end{prop}

\begin{theo}[{\bf \cite[(6.24)]{ndw} (cf.~\cite[II (1.2)]{ir}, 
\cite[(2.17)]{lodw})}]\label{theo:raycomp}
The isomorphism $(\ref{eqn:raylpi})$
induces the following isomorphism in 
${\rm D}^{\rm b}({\cal T},W_n[d]){\rm :}$
\begin{equation*}
R_n\otimes^L_{R_{\infty}}
{\mathfrak W}\Om^{\bul} \os{\sim}{\lo}
{\mathfrak W}_n\Om^{\bul}.  
\tag{3.9.1}\label{eqn:ryncp}
\end{equation*}
\end{theo}


\par
Let $Z$ be a scheme of characteristic $p>0$. 
Let ${\cal W}_n({\cal O}_Z)'$ be the obverse Witt sheaf of 
$Z$ denoted by $W_n({\cal O}_Z)''$ in {\rm \cite[\S7]{ndw}}. 
Let ${\cal B}$ be a $p$-torsion free quasi-coherent sheaf of 
commutative rings with unit elements in $\wt{Z}_{\rm zar}$ 
with a surjective morphism 
${\cal B} \lo {\cal O}_Z$ of 
sheaves of rings in $\wt{Z}_{\rm zar}$. 
Assume that ${\rm Ker}({\cal B} \lo {\cal O}_Z)=p{\cal B}$ and  
that each $\Om^i$ $(i\in {\mab N})$ 
is a quasi-coherent ${\cal B}$-module. 
Then we can endow 
${\mathfrak W}_n\Om^i$ with 
a natural ${\cal W}_n({\cal O}_Z)'$-module structure (cf.~\cite[III (1.5)]{ir}): 
for a local section $c:=(c_0, \ldots, c_{n-1})$ 
$(c_i \in{\cal O}_Z~(0\leq i \leq n-1))$ and a local section  
$\om$ of $Z^i_n$, we define $c\cdot[\om]$ as follows:
$c\cdot [\om]=
[(\sum_{j=0}^{n-1}p^j\wt{c}_j^{p^{n-j}})
\cdot \om]$, where $\wt{c}_j\in {\cal B}/p^n{\cal B}$ is  a lift of $c_j$.
We can easily check that ${\mathfrak W}_n\Om^i$ 
is a quasi-coherent ${\cal W}_n({\cal O}_Z)$-module and that 
the morphisms  $R \col {\mathfrak W}_{n+1}\Om^i \lo 
{\mathfrak W}_n\Om^i$ 
is a morphism of ${\cal W}_{n+1}({\cal O}_Z)'$-modules. 
We can easily check that ${\mathfrak W}_n\Om^i$ 
is a quasi-coherent ${\cal W}_n({\cal O}_Z)'$-module and that 
the morphism  $R \col {\mathfrak W}_{n+1}\Om^i \lo 
{\mathfrak W}_n\Om^i$ is a morphism of ${\cal W}_{n+1}({\cal O}_Z)'$-modules. 
We consider $Z_n{\mathfrak W}_1\Om^i$ 
and $B_n{\mathfrak W}_1\Om^i$ 
as ${\cal O}_Z$-submodules of $F^n_{Z*}({\mathfrak W}_1\Om^i)$. 

\begin{prop}\label{prop:fzg}
Let $F_{{\cal W}_n(Z)} \col {\cal W}_n(Z) \lo {\cal W}_n(Z)$ be 
the Frobenius endomorphism of ${\cal W}_n(Z)$. 
Then the following hold$:$
\par 
$(1)$ The following exact sequence 
\begin{align*} 
F_{{\cal W}_n(Z)*}({\mathfrak W}_{n}\Om^i)
\os{V}{\lo} {\mathfrak W}_{n+1}\Om^i
\os{F^n}{\lo} 
Z_n{\mathfrak W}_1\Om^i
\lo 0
\tag{3.10.1}\label{ali:mvee}
\end{align*}
obtained by {\rm (\ref{ali:nwee})} is an exact sequence of 
${\cal W}_{n+1}({\cal O}_Z)'$-modules. 
\par 
$(2)$  
The following exact sequence  
\begin{align*} 
{\mathfrak W}_{n+1}\Om^i
\os{F}{\lo} F_{{\cal W}_n(Z)*}({\mathfrak W}_n\Om^i)
\os{F_{{\cal W}_n(Z)*}(F^{n-1}d)}{\lo} 
B_n{\mathfrak W}_1\Om^{i+1}
\lo 0
\tag{3.10.2}\label{ali:mfee}
\end{align*}
obtained by {\rm (\ref{ali:nmee})} is an exact sequence of 
${\cal W}_{n+1}({\cal O}_Z)'$-modules. 
\end{prop}
\begin{proof} 
Set $c(n+1):=(c_0,\ldots,c_n)\in {\cal W}_{n+1}({\cal O}_Z)'$ 
$(c_i \in{\cal O}_Z~(0\leq i \leq n))$. 
Set also $c(n):=(c_0,\ldots,c_{n-1})\in {\cal W}_n({\cal O}_Z)'$.  
Let $\wt{c}_j\in {\cal B}/p^{n+1}{\cal B}$ be a lift of $c_j$. 
For a local section $\om$ of ${\mathfrak W}_m\Om^i$ $(m=n,n+1)$,  
let $\wt{\om}\in Z^i_m$ be a representative of $\om$ and 
let $[\wt{\om}]_m=\om$ be the class of $\wt{\om}$ in 
${\mathfrak W}_m\Om^i$.
We use the similar notation for $[\wt{\om}]_l$ for $l\leq m$. 
\par 
(1): Let $\om$ be a local section of 
$F_{{\cal W}_n(Z)*}({\mathfrak W}_n\Om^i)$.  
Then 
\begin{align*} 
c(n+1)\cdot V(\om)=&
[\sum_{j=0}^np^{j}\wt{c}_j^{p^{n+1-j}}p\wt{\om}]_{n+1}=
[p\sum_{j=0}^{n-1}p^j(\wt{c}_j^p)^{p^{n-j}}\wt{\om}]_{n+1}=
p[\sum_{j=0}^{n-1}p^j(\wt{c}_j^p)^{p^{n-j}}\wt{\om}]_{n+1}
\tag{3.10.3}\label{ali:mflppee}\\
&=V(c(n)\cdot \om).
\end{align*}   
This formula shows that $V$ is a morphism of 
${\cal W}_{n+1}({\cal O}_Z)'$-modules. 
\par 
Let $\om$ be a local section of ${\mathfrak W}_{n+1}\Om^i$. 
Then 
\begin{align*} 
F^{n}(c(n+1)\cdot \om)&=
[\sum_{j=0}^{n}p^j\wt{c}_j^{p^{n+1-j}}\om]_1
=[(\wt{c}^p_0)^{p^{n}}\om]_1=c_0\cdot [\om]_1. 
\tag{3.10.4}\label{ali:mflosee}
\end{align*}  
This formula shows that 
$F^n\col {\mathfrak W}_{n+1}\Om^i\lo Z_n{\mathfrak W}_1\Om^i$ 
is a morphism of ${\cal W}_{n+1}({\cal O}_Z)'$-modules. 
\par 
(2):
Let $\om$ be a local section of ${\mathfrak W}_{n+1}\Om^i$. 
Then 
\begin{align*} 
F(c(n+1)\cdot \om)=
[\sum_{j=0}^np^j\wt{c}_j^{p^{n+1-j}}\wt{\om}]_n
=[\sum_{j=0}^{n-1}p^{j}(\wt{c}^p_j)^{p^{n-j}}\wt{\om}]_n
=c(n)\cdot F(\om). 
\tag{3.10.5}\label{ali:mflepe}
\end{align*}  
\par 
Let $\om$ be a local section of $F_{{\cal W}_n(Z)*}({\mathfrak W}_n\Om^i)$.  
Then 
\begin{align*} 
F^{n-1}d(c(n)\cdot \om)&=
[p^{-n}d(\sum_{j=0}^{n-1}p^j((\wt{c}_j)^p)^{p^{n-j}})\wt{\om}]_1
=[p\sum_{j=0}^{n-1}\wt{c}_j^{p^{n+1-j}-1}d\wt{c}_j\wedge \wt{\om}
+\sum_{j=0}^{n-1}p^j\wt{c}_j^{p^{n+1-j}}p^{-n}d\wt{\om}]_1
\tag{3.10.6}\label{ali:mflee}\\
&=[(\sum_{j=0}^{n-1}p^j\wt{c}_j^{p^{n+1-j}})p^{-n}d\wt{\om}]_1=
[(\wt{c}^p_0)^{p^{n}}p^{-n}d\wt{\om}]_1=c_0\cdot F^{n-1}d \om.
\end{align*}  
This formula shows that 
$F_{{\cal W}_n(Z)*}(F^{n-1}d)\col F_{{\cal W}_n(Z)*}({\mathfrak W}_n\Om^i)
\lo B_n{\mathfrak W}_1\Om^{i+1}$ is 
a morphism of ${\cal W}_n({\cal O}_Z)'$-modules. 
\end{proof}

\begin{rema}\label{rema:omst}
(1) In \cite[I (3.11)]{idw} we can find a corresponding statement to 
(\ref{prop:fzg}). However the 
${\cal W}_{n+1}({\cal O}_Z)'$-module structures were not considered 
in [loc.~cit.]; in [loc.~cit.] only exact sequences of abelian sheaves have 
been considered. 
However the well-known relation ``$x Vy=V(Fx y)$'' implies that 
$V$ in (\ref{ali:mvee}) is compatible with 
the ${\cal W}_{n+1}({\cal O}_Z)$-structures.  
\par 
(2) Let $Z/s$ be a fine log scheme. 
The proposition (\ref{prop:fzg}) is important 
because several properties of the de Rham-Witt sheaf 
${\cal W}_n\Om^i_{Z}$ $(i\in {\mab N})$  are 
obtained by properties of 
$Z_n\Om^i_{Z}$ or $B_n\Om^{i+1}_{Z}$ (\cite{nb}). 
\end{rema}

\begin{defi}
We call the exact sequences (\ref{ali:mvee}) and (\ref{ali:mfee}) of 
${\cal W}_{n+1}({\cal O}_Z)$-modules 
the {\it log Illusie exact sequence of $(\Om^{\bul},\phi)$ in level $n$}. 
\end{defi} 

\begin{lemm}\label{lemm:nfgg}
Assume that ${\mathfrak W}_1\Om^i$ is an ${\cal O}_Z$-module of finite type 
and that $F_Z$ is a finite morphism. 
Assume that $\os{\circ}{Z}$ is a noetherian scheme. 
Then $Z_n{\mathfrak W}_1\Om^i$ and $B_n{\mathfrak W}_1\Om^i$ 
are coherent ${\cal O}_Z$-modules. 
\end{lemm}
\begin{proof} 
By the assumption, 
$F_{Z*}^n({\mathfrak W}_1\Om^i)$ is a coherent ${\cal O}_Z$-module. 
Hence $Z_n{\mathfrak W}_1\Om^i$ and $B_n{\mathfrak W}_1\Om^i$ 
are coherent ${\cal O}_Z$-modules. 
\end{proof} 

\begin{prop}[{\bf \cite[(6.12) (2)]{ndw}}]\label{prop:fg}
Let $F_Z \col Z \lo Z$ be 
the Frobenius endomorphism of $Z$. 
Assume that  
$$C^{-1}\col \Om^i_1 \os{\sim}{\lo} 
{\mathfrak W}_1\Om^i={\cal H}^i(F_{Z*}(\Om^{\bul}_1))$$ 
is an isomorphism of ${\cal O}_Z$-modules. 
If $\Om^j_1$ $(j= i-1,i)$ is 
an ${\cal O}_Z$-module of finite type and 
if $F_{Z}$ is finite, 
then ${\mathfrak W}_n\Om^i$ is a $W_n({\cal O}_Z)'$-module of finite type. 
\end{prop}

\par 
Let $s$ be a fine log scheme whose underlying scheme is ${\rm Spec}(\kap)$.  
If $Z$ is a underlying scheme of a log smooth scheme $Y$ 
of Cartier type over $s$, then ${\cal W}_n({\cal O}_Z)'={\cal W}_n({\cal O}_Y)$ 
(\cite[(7.5)]{ndw}), 
where ${\cal W}_n({\cal O}_Y)$ is a reverse Witt sheaf of $Y/s$ 
in the sense of [loc.~cit.]. 
By this identification, ${\mathfrak W}_n\Om^i$ is a quasi-coherent 
${\cal W}_n({\cal O}_Y)$-module.

\begin{prop}[{\bf \cite[(6.27) (1)]{ndw}}]
\label{prop:bodw}
Let ${\cal W}_n(s)$ and ${\cal W}(s)$ be 
the canonical lifts of $s$ over ${\cal W}_n$ and ${\cal W}$, respectively. 
Let $Y$ be a log smooth scheme of Cartier type over $s$. 
Let ${\cal Y}/{\cal W}(s)$ be a formally log smooth lift of $Y/s$ 
with a lift $\Phi \col {\cal Y} \lo {\cal Y}$ of the Frobenius endomorphism of $Y$.
Set ${\cal Y}_n:={\cal Y}\otimes_{\cal W}{\cal W}_n$ 
$(n \in {\mab Z}_{>0})$. 
Let $\Om^{\bul}_n$ be the log de Rham complex of ${\cal Y}_n/{\cal W}_n(s)$. 
Set $\Om^{\bul}:=\vpl_n\Om_n^{\bul}$.
Let $C^{-1} \col \Om^i_1 \os{\sim}{\lo}{\cal H}^i(\Om^{\bul}_1)$
be the log inverse Cartier isomorphism {\rm (\cite[(4.12) (1)]{klog1})}.
Then $(\Om^{\bul}, \Phi^*,C^{-1})$ satisfies the conditions 
$(2.1,3)\sim (2.1.7)$ for ${\cal T}=(Y_{\rm zar}, {\cal W})$.
\end{prop}

\begin{coro}\label{coro:locdw}
Let $Y$ be a log smooth scheme of Cartier type over $s$.
Let ${\cal W}_{\star}\Om^{\bul}_Y$ 
$(\star=n$ or nothing$)$ be the log de Rham-Witt complex of $Y/s$. 
Then the statements in this section with the replacement of 
${\mathfrak W}_{\star}\Om^i$ by ${\cal W}_{\star}\Om^i_Y$ hold. 
\end{coro}

The following proposition and the following corollary tells us that 
the former is a generalization of (\ref{prop:wus}): 

\begin{prop}\label{prop:seg} 
Let $({\cal W}_n\Om^{\bul}_Y)'$ be the obverse log de Rham-Witt complex of 
$Y/{\cal W}_n(s)$ defined in {\rm \cite[\S7]{ndw}} and denoted by 
$(W_n\Om^{\bul}_Y)''$ in {\rm [loc.~cit.]}.  
Let $C^{-n}\col ({\cal W}_n\Om^{\bul}_Y)'\os{\sim}{\lo} {\cal W}_n\Om^{\bul}_Y$ 
be the isomorphism of Raynaud algebras over $\kap$ 
defined in {\rm \cite[(7.0.5)]{ndw}}. 
$($In {\rm \cite[(7.5)]{ndw}} we have proved that this is an isomorphism.$)$
Then the following diagram of 
${\cal W}_{n+1}({\cal O}_Y)'$-modules and 
${\cal W}_{n+1}({\cal O}_Y)$-modules
is commutative$:$
\begin{equation*} 
\begin{CD} 
({\cal W}_n\Om^i_Y)'/F({\cal W}_{n+1}\Om^i_Y)'
@>{F^{n-1}d,~\simeq}>>B_n\Om^{i+1}_Y\\
@V{C^{-n}}V{\simeq}V @V{C^{-1}}V{\simeq}V\\
{\cal W}_n\Om^i_Y/
F{\cal W}_{n+1}\Om^i_Y@>{F^{n-1}d,~\simeq}>>B_n{\cal W}_1\Om^{i+1}_Y. 
\end{CD} 
\tag{3.17.1}\label{cd:nmkee}
\end{equation*} 
\end{prop} 
\begin{proof} 
This immediately follows from the comparison isomorphism 
$C^{-n}\col ({\cal W}_n\Om^i_Y)'\os{\sim}{\lo} 
{\cal W}_n\Om^i_Y$ (\cite[(7.5)]{ndw}) 
which is compatible with $d$'s and $F$'s. 
\end{proof} 

\begin{coro}\label{coro:nif} 
The upper horizontal isomorphism 
$F^{n-1}d \col ({\cal W}_n\Om^i_Y)'/F({\cal W}_{n+1}\Om^i_Y)'
\os{\sim}{\lo} B_n\Om^{i+1}_Y$ 
in {\rm (\ref{cd:nmkee})} for the case $i=0$ is 
equal to the isomorphism {\rm (\ref{ali:fwo})}.
\end{coro} 
\begin{proof} 
By the construction of the morphism $s_n$ in \cite[(7.0.5)]{ndw}, 
the composite morphism 
$C^{-1}\circ F^{n-1}d$ is equal to the following morphism 
\begin{align*}
{\cal W}_n({\cal O}_Y)'/F{\cal W}_{n+1}({\cal O}_Y)'\owns 
(a_0,\ldots, a_{n-1})\lom [\sum_{i=0}^{n-1}a^{p^{n-i}-1}_ida_i]
\in B_n{\cal W}_1\Om^i_Y
\end{align*} 
(see \cite[p.~251]{hk} for the definition of  the morphism $\del$). 
Because $C^{-1}(a^{p^{n-1-i}-1}_ida_i)=a^{p^{n-i}-1}_ida_i$, 
the upper horizontal morphism is equal to the morphism $d_n$. 
\end{proof}

\section{Finiteness of cohomologies of log Hodge-Witt sheaves}
\label{sec:fdlr} 
In this section we prove the main results (\ref{theo:dw}) and (\ref{theo:ht}).

\par 
Let $s$ be as in the Introduction. 
Let $X/s$ be a proper log smooth scheme of Cartier type.

\par 
The following is a log version of \cite[II (2.2), (3.1)]{ir}. 

\begin{theo}\label{theo:e2} 
The $E_2$-terms of the slope spectral sequence 
\begin{align*} 
E_1^{ij}=H^j(X,{\cal W}\Om^i_X)\Lo H^{i+j}_{\rm crys}(X/{\cal W}(s))
\end{align*}  
are finitely generated ${\cal W}$-modules. 
\end{theo} 
\begin{proof} 
By using (\ref{theo:raycomp}) and (\ref{coro:locdw}), 
the proof is the same as that of \cite[II (2.2), (3,1)]{ir}. 
\end{proof}


\begin{theo}\label{theo:fw}
Let $q$ and $i$ be integers. 
Assume that $\dim_{\kap}H^{q-1}(X,B_n\Om^{i+1}_{X/s})$ 
is bounded for all $n$. 
Then the differential 
$d\col H^q(X,{\cal W}\Om^i_X)\lo H^q(X,{\cal W}\Om^{i+1}_X)$ is zero.  
Consequently 
$$H^q(X,{\cal W}\Om^i_X)/dH^q(X,{\cal W}\Om^{i-1}_X)$$ 
is a finitely generated ${\cal W}$-module. 
\end{theo}
\begin{proof} 
(cf.~the proof of \cite[(5.1)]{j})
Recall that  
$${}_FH^q(X,{\cal W}\Om^i_X):=
{\rm Ker}(F\col H^q(X,{\cal W}\Om^i_X)
\lo H^q(X,{\cal W}\Om^i_X)).$$  
By the log version of \cite[II (3.8)]{ir} it suffices to prove that 
$\dim_{\kap}({}_FH^q(X,{\cal W}\Om^i_X))<\infty$.
By the exact sequence 
\begin{align*} 
0\lo {\cal W}\Om^i_X \os{F}{\lo} {\cal W}\Om^i_X\lo {\cal W}\Om^i_X/F\lo 0, 
\tag{4.2.1}\label{ali:fwox} 
\end{align*} 
we have the following surjection  
\begin{align*} 
H^{q-1}(X,{\cal W}\Om^i_X/F{\cal W}\Om^i_X)
\lo {}_FH^q(X,{\cal W}\Om^i_X). 
\tag{4.2.2}\label{ali:wffiox} 
\end{align*} 
It suffices to prove that 
$H^{q-1}(X,{\cal W}\Om^i_X/F{\cal W}\Om^i_X)$ 
is a finitely generated ${\cal W}$-module.
Because 
$H^{q-1}(X,{\cal W}\Om^i_X/F{\cal W}\Om^i_X)
=\vpl_nH^{q-1}(X,{\cal W}_n\Om^i_X/F{\cal W}_{n+1}\Om^i_X)$, 
it suffices to prove that 
$H^{q-1}(X,{\cal W}_n\Om^i_X/F{\cal W}_{n+1}\Om^i_X)$'s 
are finite dimensional $\kap$-vector spaces of 
bounded dimensions for all $n$'s. 
By (\ref{prop:fkervcoker}) and (\ref{coro:locdw}),  
$$H^{q-1}(X,{\cal W}_n\Om^i_X/F{\cal W}_{n+1}\Om^i_X)=
H^{q-1}(X,B_n{\cal W}_1\Om^{i+1}_X)
\simeq H^{q-1}(X,B_n\Om^{i+1}_{X/s}).$$ 
Hence $H^{q-1}(X,{\cal W}_n\Om^i_X/F{\cal W}_{n+1}\Om^i_X)$'s  
are finite dimensional $\kap$-vector spaces of 
bounded dimensions for all $n$'s by the assumption.
Now we see that $H^q(X,{\cal W}\Om^i_X)/dH^q(X,{\cal W}\Om^{i-1}_X)=E_2^{0q}$ 
is finitely generated by (\ref{theo:e2}).  
\end{proof}

\begin{theo}\label{theo:fsb0}
Assume that $\os{\circ}{X}$ is quasi-$F$-split. 
Then 
$\dim_{\kap}H^q(X,B_n\Om^1_{X/s})$ 
is bounded for all $n$ and for all $q$.
\end{theo} 
\begin{proof} 
(cf.~\cite[(2.4.1)]{jr})
Let $n$ be a positive integer. 
Push out the exact sequence (\ref{ali:wnox}) for the case $Y=X$ 
by the morphism $R^{n-1}\col {\cal W}_n({\cal O}_X)\lo {\cal O}_X$. 
Then we have the following exact sequence of ${\cal O}_X$-modules: 
\begin{align*} 
0\lo {\cal O}_X \lo {\cal E}_n \lo B_n\Om^1_{X/s}\lo 0,
\tag{4.3.1}\label{ali:wrnyex} 
\end{align*}  
where 
${\cal E}_n
:={\cal O}_X\oplus_{{\cal W}_n({\cal O}_X),F}F_*({\cal W}_n({\cal O}_X))$. 
Set $h:=h_F(\os{\circ}{X})<\infty$. 
For $n=h$, the exact sequence (\ref{ali:wrnyex}) is split. 
It is easy to check that the exact sequence (\ref{ali:wrnyex}) is split 
for $n\geq h$ (\cite[(8.2) (2)]{ny}). 
Hence 
\begin{align*} 
H^q(X,{\cal E}_n)=H^q(X,{\cal O}_X)\oplus 
H^q(X,B_n\Om^1_{X/s})
\tag{4.3.2}\label{ali:dew}
\end{align*} 
for $n\geq h$. 
In particular, 
\begin{align*}
{\rm dim}_{\kap}H^q(X,{\cal E}_n)
={\rm dim}_{\kap}H^q(X,{\cal O}_X)+
{\rm dim}_{\kap}H^q(X,B_n\Om^1_{X/s}). 
\tag{4.3.3}\label{ali:dw}
\end{align*} 
\par 
Following \cite[(4.1), (4.2)]{yoh} in the trivial logarithmic case, 
consider the following diagram 
\begin{equation*} 
\begin{CD} 
@. @. @.0\\
@. @. @. @VVV \\
@.0@. @.F_*(B_{n-1}\Om^1_{X/s})@. @.\\
@. @VVV @. @VVV \\
0@>>>{\cal O}_{X} @>>> {\cal E}_n@>>> B_n\Om^1_{X/s}@>>> 0\\
@. @| @VVV @VV{C^{n-1}}V \\
0@>>> {\cal O}_{X}
@>{\subset}>> {\cal E}_1=F_*({\cal O}_X) 
@>{d}>>B_1\Om^1_{X/s}@>>> 0\\
@. @VVV @VVV @VVV \\
@.0@. 0@. 0
\end{CD}
\tag{4.3.4}\label{cd:dnbx}
\end{equation*} 
of ${\cal O}_X$-modules.
Here we have used the commutative diagram (\ref{cd:dowm}). 
Using the snake lemma, we obtain the following exact sequence: 
\begin{align*} 
0\lo F_*(B_{n-1}\Om^1_{X/s}) \lo {\cal E}_n \lo F_*({\cal O}_X) \lo 0. 
\tag{4.3.5}\label{ali:wrbyex} 
\end{align*}
Let ${\cal F}$ be a quasi-coherent ${\cal O}_X$-module. 
Then 
$H^q(X,F_*({\cal F}))=H^q(X,{\cal F})$ 
with $\kap$-module structure obtained by 
the Frobenius automorphism $\sig$ of $\kap$ 
since $\os{\circ}{F}$ is finite.  
Hence, by (\ref{ali:wrbyex}), 
\begin{align*}
{\rm dim}_{\kap}H^q(X,{\cal E}_n)\leq 
{\rm dim}_{\kap}H^q(X,B_{n-1}\Om^1_{X/s})+
{\rm dim}_{\kap}H^q(X,{\cal O}_X). 
\tag{4.3.6}\label{ali:ffw}
\end{align*}  
The equality (\ref{ali:dw}) and the inequality (\ref{ali:ffw}) imply 
that 
\begin{align*}
{\rm dim}_{\kap}H^q(X,B_n\Om^1_{X/s})\leq 
{\rm dim}_{\kap}H^q(X,B_{n-1}\Om^1_{X/s}).
\tag{4.3.7}\label{ali:dbw}
\end{align*}  
This implies that 
\begin{align*} 
{\rm dim}_{\kap}H^q(X,B_n\Om^1_{X/s})
\leq 
{\max}\{{\rm dim}_{\kap}H^q(X,B_m\Om^1_{X/s})~\vert~1\leq m\leq h, 
0\leq q\leq \dim \os{\circ}{X}\}.
\tag{4.3.8}\label{ali:dbbw}
\end{align*}  
\end{proof} 

\begin{coro}\label{coro:jjfs}
{\rm (\ref{theo:dw})} holds. 
\end{coro}
\begin{proof}
This follows from (\ref{theo:fw}) and (\ref{theo:fsb0}). 
\end{proof}

\begin{rema}\label{rema:sei}
To prove (\ref{theo:dw}), we have followed the argument in the proof of 
\cite[(5.1)]{j}.  However Serre has proved the following in 
\cite[p.~510, Corollaire 1]{smxco}: 
$H^q(X,{\cal W}({\cal O}_X))$ is a finitely generated ${\cal W}$-module 
if and only if the dimension of 
$\vpl_nH^q(X,{\cal W}_n({\cal O}_X)/F{\cal W}_n({\cal O}_X))$ 
over $\kap$ is finite. 
It is clear that, if the dimension of 
$H^q(X,{\cal W}_n({\cal O}_X)/F{\cal W}_n({\cal O}_X))$ 
over $\kap$ is bounded for $n$'s, 
then the dimension of 
$\vpl_nH^q(X,{\cal W}_n({\cal O}_X)/F{\cal W}_n({\cal O}_X))$
is finite. 
If one would like to prove only (\ref{theo:dw}), 
only this Serre's result,  the log Serre exact sequence 
(\ref{ali:wnox}) and (\ref{theo:fsb0}) are enough. 
\end{rema}

\begin{coro}\label{coro:vp}
Assume that $\os{\circ}{X}$ is quasi-$F$-split. 
Set $H^q:=H^q(X,{\cal W}({\cal O}_X))$ and 
${}_VH^q:={\rm Ker}(V\col H^q\lo H^q)$ 
and 
${}_pH^q:={\rm Ker}(p\col H^q\lo H^q)$.  
Then the subvector space ${}_VH^q$ in ${}_pH^q$
has finite codimension in ${}_pH^q$. 
\end{coro}
\begin{proof} 
Consider the following exact sequence 
\begin{align*} 
0\lo {}_VH^q\lo {}_pH^q
\lo {}_pH^q/{}_VH^q\lo 0. 
\end{align*} 
Since the morphism 
$V\col {}_pH^q/{}_VH^q\lo {}_FH^q$ is injective and 
${\rm dim}_{\kap}({}_FH^q)< \infty$, 
we have the desired inequality 
${\rm dim}_{\kap}({}_pH^q/{}_VH^q)<\infty$. 
\end{proof}

\begin{coro}\label{coro:j}
The following hold$:$
\par 
$(1)$ {\rm (\ref{coro:bb})} holds. 
\par 
$(2)$ {\rm (\ref{coro:adge1})} holds. 
\par 
$(3)$  
{\rm (\ref{coro:dge1})} holds. 
\par 
$(4)$ {\rm (\ref{coro:pnc})} holds.  
\end{coro}
\begin{proof} 
(1): Set $d:=\dim \os{\circ}{X}$. 
By the following exact sequence 
\begin{align*} 
0\lo B_{\infty}\Om^1_{X/s}\lo \Om^1_{X/s}\lo 
\Om^1_{X/s}/B_{\infty}\Om^1_{X/s}\lo 0, 
\end{align*} 
we obtain the following exact sequence 
\begin{align*} 
0& \lo H^0(X,B_{\infty}\Om^1_{X/s})\lo 
H^0(X,\Om^1_{X/s})\lo H^0(X,\Om^1_{X/s}/B_{\infty}\Om^1_{X/s})
\lo \cdots  \\
&\lo H^d(X,B_{\infty}\Om^1_{X/s})\lo 
H^d(X,\Om^1_{X/s})\lo H^d(X,\Om^1_{X/s}/B_{\infty}\Om^1_{X/s})
\lo 0. 
\end{align*} 
Hence it suffices to prove that 
$H^q(X,B_{\infty}\Om^1_{X/s})$ $(q\in {\mab N})$ is finite dimensional. 
Consider $B_n\Om^1_{X/s}$ as a sheaf of 
$f^{-1}(\kap)$-submodules of $\Om^1_{X/s}$, where $f\col X\lo s$ 
is the structural morphism. 
We denote this resulting sheaf by $\sig_*^{-n}(B_n\Om^1_{X/s})$. 
Because $H^q(X,B_{\infty}\Om^1_{X/s})=
\vil H^q(X,\sig_*^{-n}(B_n\Om^1_{X/s}))$ 
and because 
$\dim_{\kap}H^q(X,\sig_*^{-n}(B_n\Om^1_{X/s}))
=\dim_{\kap}H^q(X,B_n\Om^1_{X/s})$, 
we see that $\dim_{\kap}H^q(X,B_{\infty}\Om^1_{X/s})<\infty$ by 
(\ref{theo:fsb0}). 
\par 
(2): The proof is the same as that of \cite[II (3.14)]{idw}. 
\par 
(3): In \cite[(6.1)]{j} Joshi has proved that 
$Y/\kap$ is of Hodge-Witt type if and only if 
$H^q(Y,{\cal W}({\cal O}_Y))$ $(q\in {\mab N})$ is 
a finitely generated ${\cal W}$-module. 
Hence (3) follows from (\ref{coro:jjfs}). 
\par 
(4): (4) follows from (3) and \cite[II (4.1)]{konp}. 
\end{proof}

\begin{exem}
Let $Y/\kap$ be a $K3$-surface with finite second Artin-Mazur height. 
Then $H^2(Y,{\cal W}({\cal O}_Y))$ is a finitely generated ${\cal W}$-module.  
(In \cite[p.~653]{idw} Illusie has already proved that they are of Hodge-Witt type.) 
\par 
More generally, let $Y/\kap$ be a $d$-dimensional Calabi-Yau variety 
with finite $d$-th Artin-Mazur height $h$. 
Then $H^d(Y,{\cal W}({\cal O}_Y))\simeq {\cal W}^{\oplus h}$.
Because $H^0(Y,{\cal W}({\cal O}_Y))={\cal W}$ and 
$H^q(Y,{\cal W}({\cal O}_Y))=0$ $(q\not=0, d)$, 
$H^q(Y,{\cal W}({\cal O}_Y))$ is a finitely generated ${\cal W}$-module for any $q$.
Consequently, if $d=3$, then $Y/\kap$ is of Hodge-Witt type by \cite[(6.1)]{j}. 
In particular, 
\begin{align*} 
H^q_{\rm crys}(Y/{\cal W})=
\bigoplus_{i+j=q}H^j(Y,{\cal W}\Om^i_Y) 
\quad (q\in {\mab N}). 
\end{align*}
\end{exem}

\begin{rema}\label{rema:fdst} 
(\ref{coro:bb}) is very important in 
Bloch-Stienstra's theory in (\cite{stjft}, \cite{stjf}). 
\par 
Let us recall their theory.  
\par 
Let $Y$ be a proper smooth scheme over $\kap$. 
Let ${\cal K}_{i,Y}$ $(i\in {\mab Z}_{\geq 1})$ 
be the sheafification of the following presheaf of abelian groups 
on $Y$: $U\lom K_i(\Gam(U,{\cal O}_Y))$,
where $U$ is an open subscheme of $Y$ and 
$K_i$ means the $i$-th Quillen's $K$-group. 

\par 
Let us consider the following inductive system
\begin{align*} 
{\cal W}_1({\cal O}_Y)\os{V}{\lo} {\cal W}_2({\cal O}_Y)\os{V}{\lo} \cdots 
\os{V}{\lo} {\cal W}_n({\cal O}_Y) \os{V}{\lo} 
{\cal W}_{n+1}({\cal O}_Y) \os{V}{\lo} \cdots. 
\end{align*} 
and set $\us{\lo}{\cal W}({\cal O}_Y):=\vil_{n} {\cal W}_n({\cal O}_Y)$.  
For $m, n\in {\mab Z}_{\geq 1}$, 
we consider a morphism 
$\partial_n \col \us{\lo}{\cal W}({\cal O}_Y)\lo {\cal W}_n\Om^1_{Y}$ 
defined by the following 
\begin{equation*} 
\partial_n\vert_{{\cal W}_m({\cal O}_Y)} := 
\begin{cases} 
-dV^{n-m} & {\rm if}~n\geq m, \\
-F^{m-n}d & {\rm if}~n\leq m.
\end{cases}
\end{equation*}  
Here ${\cal W}_n\Om^1_{Y}$ is 
the de Rham-Witt sheaf defined in \cite{idw}. 
(Note that the de Rham-Witt sheaf is isomorphic to 
the sheaf of $p$-typical curves defined in \cite{bl} (\cite[I, 5]{idw})). 
Then the projection 
$R\col {\cal W}_{n+1}\Om^1_{Y} \lo  {\cal W}_n\Om^1_{Y}$ 
induces the surjective morphism 
$\partial_{n+1}\us{\lo}{\cal W}({\cal O}_Y)\lo 
\partial_n\us{\lo}{\cal W}({\cal O}_Y)$. 
Set ${\cal W}\Om^1_{Y}/\partial \us{\lo}{\cal W}({\cal O}_Y)
:=
\vpl_n{\cal W}_n\Om^1_{Y}/\partial_n\us{\lo}{\cal W}({\cal O}_Y)$ 
and 
$\partial \us{\lo}{\cal W}({\cal O}_Y):=
{\rm Ker}({\cal W}\Om^1_{Y}\lo 
{\cal W}\Om^1_{Y}/\partial \us{\lo}{\cal W}({\cal O}_Y))$. 
Let $D(\kap)$ be the Dieudonn\'{e} ring of $\kap$. 
It is well-known 
that ${\cal W}\Om^1_{Y}$ and $\partial \us{\lo}{\cal W}({\cal O}_Y)$ 
are sheaves of left $D(\kap)$-module and 
$\partial \us{\lo}{\cal W}({\cal O}_Y)$ is a subsheaf of 
left $D(\kap)$-modules of ${\cal W}\Om^1_{Y}$. 
By replacing the roles of $F$ and $V$, 
$\partial \us{\lo}{\cal W}({\cal O}_Y)$ is a subsheaf of 
right $D(\kap)$-modules of ${\cal W}\Om^1_{Y}$. 
\par 
Let $A$ be an artinian $\kap$-algebra. 
Set $\wh{\rm C}{\rm K}_i(A):={\rm Ker}(K_i(A[x])\os{x\lom 0}{\lo} K_i(A))$. 
Let $F_m$ and $V_m$ $(m\in {\mab Z}_{\geq 1})$ 
be the standard operators on 
$\wh{\rm C}{\rm K}_i(A)$ induced by those on $K_i(A[x])$. 
Set $e:=\sum_{(m,p)=1}\dfrac{\mu(m)}{m}V_mF_m$. 
Set ${\rm T}\wh{\rm C}{\rm K}_i(A):=e\wh{\rm C}{\rm K}_i(A)$ 
and $\wh{W}(A):={\rm T}\wh{\rm C}{\rm K}_1(A)$. 
The last group is a left $D(\kap)$-module. 
The following functor arises in Bloch-Stienstra's theory: 
\begin{align*} 
{\rm Bl}^q_{Y/\kap}(A):=
H^q(Y,
{\cal W}\Om^1_{Y}/\partial \us{\lo}{\cal W}({\cal O}_Y)\otimes_{D(\kap)}\wh{W}(A))
\in ({\rm Ab})
\end{align*}
for artinian local $\kap$-algebras $A$'s with residue fields $\kap$. 
Assume that $Y/\kap$ is proper and smooth. 
Then the tangent space $T({\rm Bl}^q_{Y/\kap})$ of this functor 
is equal to $H^q(Y,\Om^1_{Y/\kap}/B_{\infty}\Om^1_{Y/\kap})$ 
(\cite[IV (3.16)]{stjft}). 
Hence we obtain the following corollary by (\ref{coro:bb}): 
\end{rema} 

\begin{coro}\label{coro:nm} 
Let the notations be as above. Assume that $Y$ is quasi-$F$-split. 
Then $T({\rm Bl}^q_{Y/\kap})$ is a finite dimensional $\kap$-vector space. 
\end{coro}


By the proof of (\ref{theo:fsb0}), we obtain the following:

\begin{coro}\label{coro:nqs}
{\rm (\ref{theo:ht})} holds. 
\end{coro}
\begin{proof} 
Assume that $h_F(\os{\circ}{X})<\infty$, 
By (\ref{eqn:pkdefpw}) and (\ref{ali:dbw}),  
$${\rm min}\{n, h^q(\os{\circ}{X}/\kap)-1\}\leq 
{\rm min}\{n-1, h^q(\os{\circ}{X}/\kap)-1\}$$ 
for $n\geq h_F(\os{\circ}{X})$. 
This implies that $h^q(\os{\circ}{X}/\kap)\leq h_F(\os{\circ}{X})$. 
\par 
If $h_F(\os{\circ}{X})=\infty$, then there is nothing to prove. 
\end{proof}

\begin{rema}\label{rema:df}
It seemed to me at first that the proof of (\ref{coro:nqs}) 
was considerably mysterious because we do not consider 
${\rm Ext}^1_X(B_n\Om^1_{X/s},{\cal O}_X)$ at all 
nor do not use the log Serre duality of Tsuji (\cite[(2.21)]{tsp}); 
the decomposition (\ref{ali:dew}) enables us to obtain (\ref{coro:nqs}).  
\end{rema}

\par
More generally we would like to ask the following: 

\begin{prob}[{\bf Inequality problem between Artin-Mazur heights and 
a Yobuko height}]\label{prob:ine}
Let $Z/\kap$ be a proper geometrically connected scheme. 
Let $q$ be a nonnegative integer. 
If $h_F(Z)<\infty$, 
then does the following inequality 
\begin{align*}
{\rm rank}_{\cal W}\{H^q(Z,{\cal W}({\cal O}_Z))/({\rm torsion})\}\leq h_F(Z) 
\tag{4.13.1}\label{ali:qzf}
\end{align*}
hold?
If $\dim_{\kap}{}_pH^q(Z,{\cal W}({\cal O}_Z))=\infty$, 
then $h_F(Z)=\infty$?   
Here ${}_pH^q(Z,{\cal W}({\cal O}_Z))$ is the subgroup of $p$-torsion elements 
of $H^q(Z,{\cal W}({\cal O}_Z))$. 
\end{prob} 

\parno 
If the answer for this problem is Yes, 
if $\Phi^q_{Z/\kap}$ is representable and 
if $h_F(Z)<\infty$, then $h^q(Z/\kap)\leq h_F(Z)$. 

The following is a generalization of \cite[(11.4)]{j}: 

\begin{coro}\label{coro:jg}
{\rm (\ref{coro:c})} holds. 
\end{coro} 
\begin{proof} 
By (\ref{coro:j}) (3) in the case where  
$\kap$ is an algebraically closed field, 
the induced morphism by the derivative 
$H^j(Y_{\ol{\kap}},{\cal W}\Om^i_{Y_{\ol{\kap}}})
\lo H^j(Y,{\cal W}\Om^{i+1}_{Y_{\ol{\kap}}})$ 
$(i,j\in {\mab N})$ is zero. 
In particular, the induced morphism by the derivative 
$H^2(Y_{\ol{\kap}},{\cal W}\Om^1_{Y_{\ol{\kap}}})
\lo H^2(Y,{\cal W}\Om^2_{Y_{\ol{\kap}}})$ is zero. 
Hence (\ref{coro:jg}) follows from \cite[III (4.7)]{gs}. 
\end{proof}

We recall the following theorem due to Yobuko. 

\begin{theo}[{\bf \cite[(3.5)]{y}}]\label{theo:yob}
Let $Y$ be a Calabi-Yau variety of pure dimension $d$. 
Then $h^d(Y/\kap)=h_F(Y)$. 
\end{theo}

\parno 
In fact we generalized this theorem in \cite{ny}: 
\begin{theo}[{\bf \cite[(10.1)]{ny}}]\label{theo:cyh} 
Let $X$ be a proper log smooth, integral
and saturated log scheme over $s$ of 
pure dimension $d$. Assume that $X/s$ is of Cartier type and of vertical type 
$($see {\rm \cite[(6.3)]{ny}} for the definition of the vertical type$)$. 
Assume also that the following three conditions hold$:$
\par 
$({\rm a})$ $H^{d-1}(X,{\cal O}_X)=0$ if $d\geq 2$, 
\par 
$({\rm b})$ $H^{d-2}(X,{\cal O}_X)=0$ if $d\geq 3$, 
\par 
$({\rm c})$ $\Omega^d_{X/s}\simeq {\cal O}_X$. 
\parno 
Then $h_F(\os{\circ}{X})=h^d(\os{\circ}{X}/\kap)$. 
\end{theo}

\begin{rema}\label{rema:fs}
Let $X/s$ be as in (\ref{theo:cyh}). 
\par 
(1) By (\ref{theo:cyh}) we see that 
$h^d(\os{\circ}{X}/\kap)$ is independent of 
the choice of the structural morphism $\os{\circ}{X}\lo \kap$; 
it depends only on $\os{\circ}{X}$. 
\par 
(2) 
Let the notations be in (\ref{theo:cyh}). 
By the equality $h_F(\os{\circ}{X}/\kap)=h^d(\os{\circ}{X}/\kap)$, 
$X$ is $F$-split if and only if $h^d(\os{\circ}{X}/\kap)=1$. 
\end{rema}

\begin{coro}
{\rm (\ref{coro:dgtfye1})} holds. 
\end{coro} 
\begin{proof}
This follows from (\ref{coro:j}) (3) and \cite[IV (4.7), (4.8)]{ir}. 
\end{proof} 

The following is a generalization of \cite[(11.3)]{j}. 

\begin{coro}\label{coro:tr}
Let the notations and the assumptions be as in {\rm (\ref{theo:cyh})}. 
Assume that the log structures of $s$ and $X$ are trivial, 
that is, $X$ is a proper smooth scheme over $\kap$. 
Assume that $h_F(\os{\circ}{X})<\infty$. 
Let $q\geq 2$ be an integer and 
assume that $X$ is of pure dimension $2q-1$.
Let $l\not=p$ be a prime number. 
Let $A^q(X_{\ol{\kap}})$ be the subgroup of 
${\rm CH}^q(X_{\ol{\kap}})$ generated by 
cycles which are algebraically equivalent to $0$. 
Let $A^q(X_{\ol{\kap}})\{l\}$ be the $l$-primary torsion part of 
$A^q(X_{\ol{\kap}})$. 
Then the following restriction of the $l$-adic Abel-Jacobi map of Bloch to 
$A^q(X_{\ol{\kap}})\{l\}$ 
\begin{align*} 
A^q(X_{\ol{\kap}})\{l\}\lo H^{2q-1}(X_{\ol{\kap}},{\mab Q}_l/{\mab Z}_l(q))
\tag{4.19.1}\label{ali:chx}
\end{align*} 
is not surjective. 
\end{coro} 
\begin{proof} 
By the equality 
$h_F(\os{\circ}{X}/\kap)=h^{2q-1}(\os{\circ}{X}/\kap)$, 
the Dieudonn\'{e} module of the Artin-Mazur formal group 
$\Phi^{2q-1}(\os{\circ}{X}/\kap)$ 
is a free ${\cal W}$-modules of finite rank $1\leq h_F(\os{\circ}{X}/\kap)<\infty$. 
Because this module is isomorphic to 
$H^{2q-1}(X,{\cal W}({\cal O}_X))$, 
$H^{2q-1}(X,{\cal W}({\cal O}_X))\otimes_{\cal W}K_0\not=0$. 
Hence the slopes of 
$H^{2q-1}(X,{\cal W}({\cal O}_X))\otimes_{\cal W}K_0$ is not contained 
in $[q-1,q]$ since $q-1\geq 1$. 
By \cite[(3.4)]{suwa} (as in \cite[(11.2)]{j}), 
the $l$-adic Abel-Jacobi map of Bloch (\ref{ali:chx}) is not surjective.  
\end{proof}

\par 
We can generalize the Yobuko height as follows. 
\par 
Let $i$ be a nonnegative integer. 
Then we have the following exact sequence
\begin{align*} 
0\lo F{\cal W}_{n+1}\Om^i_{X}\lo F_{{\cal W}_n(X)*}({\cal W}_n\Om^i_{X})
\os{F^{n-1}d}{\lo} B_n\Om^{i+1}_{X/s}\lo 0 \quad (n\in {\mab Z}_{>0})
\tag{4.19.2}\label{ali:in}
\end{align*} 
of ${\cal W}_n({\cal O}_X)$-modules. 
Consider the push-out of the exact sequence (\ref{ali:in}) by the projection  
$F{\cal W}_{n+1}\Om^i_{X}\lo F{\cal W}_2\Om^i_{X}$ 
and let 
\begin{align*} 
0\lo F{\cal W}_2\Om^i_{X}\lo {\cal E}^i_n
{\lo} B_n\Om^{i+1}_{X/s}\lo 0 \quad (n\in {\mab Z}_{>0}).
\tag{4.19.3}\label{ali:ipn}
\end{align*} 
be the resulting exact sequence of ${\cal O}_X$-modules. 
We say that $X/s$ has {\it height} $h<\infty$ at $i$ if (\ref{ali:ipn}) 
is split for $\forall n\geq h$. If (\ref{ali:ipn}) is not split, then we set $h=\infty$. 
(Note that, by (\ref{cd:nmwee}), if (\ref{ali:ipn}) is split for some 
$n\in {\mab Z}_{\geq 0}$,  then (\ref{ali:ipn}) for any $m\geq n$ is split.)
We denote $h$ by $h_F^i(X/s)$. 
It is easy to check that $h_F^0(X/s)=h_F(\os{\circ}{X})$. 
\par 
Assume that $h_F^i(X/s)<\infty$. 
Then, by the same proof as that of (\ref{theo:fsb0}), 
$H^q(X,B_n\Om^{i+1}_{X/s})$ is bounded for all $n$. 
Indeed, we have the following commutative diagram of exact sequences
\begin{equation*} 
\begin{CD} 
0@>>>F{\cal W}_2\Om^i_X @>>> {\cal E}^i_n@>>> B_n\Om^{i+1}_{X/s}@>>> 0\\
@. @| @VVV @VV{C^{n-1}}V \\
0@>>> F{\cal W}_2\Om^i_X
@>{\subset}>> {\cal E}^i_1=F_*(\Om^i_{X/s}) 
@>{d}>>B_1\Om^{i+1}_{X/s}@>>> 0\\
\end{CD}
\tag{4.19.4}\label{cd:dnyb}
\end{equation*} 
of ${\cal O}_X$-modules 
and the following exact sequence 
\begin{align*} 
0\lo   B_{n-1}\Om^{i+1}_{X/s}\lo {\cal E}^i_n\lo F_*(\Om^i_{X/s}) \lo 0 
\tag{4.19.5}\label{cd:dbeyb}
\end{align*} 
of ${\cal O}_X$-modules.  
Hence, by the same proof as that of (\ref{theo:fsb0}), 
we obtain the following inequality
\begin{align*} 
\dim_{\kap}H^q(X,B_n\Om^{i+1}_{X/s})\leq &
\dim_{\kap}H^q(X,B_{n-1}\Om^{i+1}_{X/s})
\\
&+\dim_{\kap}H^q(X,\Om^i_{X/s})-\dim_{\kap}H^q(X,F{\cal W}_2\Om^i_X)
\end{align*} 
for $n\geq h$. 
In this way, 
we see that $H^q(X,B_n\Om^{i+1}_{X/s})$ is bounded for all $n$. 
By (\ref{theo:fw}) the differential 
$H^q(X,{\cal W}\Om^i_X)\lo H^q(X,{\cal W}\Om^{i+1}_X)$ is zero.  
Consequently 
$H^q(X,{\cal W}\Om^i_X)/dH^q(X,{\cal W}\Om^{i-1}_X)$
is a finitely generated ${\cal W}$-module. 
\par 
If the log structures of $s$ and $X$ are trivial and if $h_F^1(X/s)<\infty$, 
then the differential 
$H^2(X_{\ol{\kap}},{\cal W}\Om^1_{X_{\ol{\kap}}})\lo 
H^2(X_{\ol{\kap}},{\cal W}\Om^2_{X_{\ol{\kap}}})$ is zero.  
Hence, by \cite[II (3.7)]{gs} ${\rm CH}^2(X_{\ol{\kap}})\{p\}$ 
is of finite cotype. 
Hence we obtain the following: 

\begin{theo}\label{theo:gn}
Let $X/\kap$ be a proper smooth scheme. 
If $h^1_F(X/\kap)<\infty$, then ${\rm CH}^2(X_{\ol{\kap}})\{p\}$ 
is of finite cotype. 
\end{theo}

\section{Upper bounds of heights of Artin-Mazur formal groups}\label{sec:ham} 
Let $X/s$ be as in the beginning of the previous section. 
In this section we give a upper bound of the height of 
the Artin-Mazur formal group $\Phi^q_{\os{\circ}{X}/\kap}$ $(q\in {\mab N})$ 
by using the dimensions of log Hodge cohomologies of $X/s$.
This is a much more general upper bound than  
Katsura and Van der Geer's upper bound 
for the Artin-Mazur formal group of 
a Calabi-Yau variety over $\kap$ 
(\cite[(2.4)]{vgkht}). 
To give the upper bound, we use 
(\ref{eqn:rfvn}) and (\ref{prop:bodw}) in \S\ref{sec:rrafr}.  
The arguments in this section are influenced 
by the arguments in \cite[II (4.1)$\sim$(4.6)]{idw}.

\begin{theo}\label{theo:htb}
Let $q$ and $i$ be nonnegative integers. 
Assume that the operator 
$$F\col H^{h}(X,{\cal W}\Om^i_X)\lo H^h(X,{\cal W}\Om^i_X) \quad (h=q,q+1)$$ 
is injective. 
Furthermore, assume 
that the operator  
$$dV \col H^q(X,{\cal W}\Om^j_X)\lo H^q(X,{\cal W}\Om^{j+1}_X)$$  
is zero for $j=i-1$ and $j=i$. 
Then there exists the following diagram 
\begin{equation*} 
\begin{CD} 
0 @>>>  H^q(X,{\cal W}\Om^{i}_X)/V@>{F}>> H^q(X,{\cal W}\Om^{i}_X)/p
@>{\rm proj}.>>
H^q(X,{\cal W}\Om^{i}_X)/F
@>>> 0\\
@. @V{\bigcap}VV@. @VV{\simeq}V\\
@. H^q(X,{\cal W}\Om^{i}_X/V) @. @. H^q(X,{\cal W}\Om^{i}_X/F)\\
@. @V{\bigcap}VV@. @AAA\\
@. H^{q}(X,\Om^{i}_{X/s})  @. @. H^{q-1}(X,\Om^{i+1}_{X/s}), 
\end{CD} 
\tag{5.1.1}\label{eqn:xxs}
\end{equation*} 
where the morphism 
$H^{q-1}(X,\Om^{i+1}_{X/s})\lo H^q(X,{\cal W}\Om^{i}_X/F)$ is 
constructed in the proof of this theorem and it is surjective. 
\end{theo}
\begin{proof} 
By the first assumption, 
we have the following exact sequence 
\begin{align*} 
0\lo H^q(X,{\cal W}\Om^i_X)/V
\os{F}{\lo} H^q(X,{\cal W}\Om^i_X)/p
\lo H^q(X,{\cal W}\Om^i_X)/F\lo 0. 
\tag{5.1.2}\label{ali:wfvnii}
\end{align*} 
By (\ref{eqn:rfvn}) for the case $r=1$ and (\ref{coro:locdw}), we have the following exact sequence 
\begin{align*} 
0\lo {\cal W}\Om^j_X/F
\os{dV}{\lo} {\cal W}\Om^{j+1}_X/V\lo \Om^{j+1}_{X/s}\lo 0. 
\tag{5.1.3}\label{ali:wfsfnii}
\end{align*} 
Here we have used the log inverse Cartier isomorphism 
$C^{-1}\col \Om^{j+1}_{X/s} \os{\sim}{\lo} {\cal H}^{j+1}(\Om^{\bul}_{X/s})$. 
Hence we have the following exact sequence: 
\begin{align*} 
\cdots & \lo H^{q-1}(X,\Om^{j+1}_{X/s})\lo 
H^{q}(X,{\cal W}\Om^{j}_X/F)
\os{dV}{\lo} H^{q}(X,{\cal W}\Om^{j+1}_X/V)
\tag{5.1.4}\label{ali:wmnii}\\
&\lo 
H^{q}(X,\Om^{j+1}_{X/s})\lo \cdots. 
\end{align*} 
By the exact sequence (\ref{ali:fwox})  
and the first assumption, 
we have the following isomorphism 
\begin{align*} 
H^q(X,{\cal W}\Om^i_X)/F\os{\sim}{\lo} 
H^q(X,{\cal W}\Om^i_X/F). 
\tag{5.1.5}\label{ali:wooix} 
\end{align*} 
By the exact sequence 
\begin{align*} 
0\lo {\cal W}\Om^{j+1}_X \os{V}{\lo} {\cal W}\Om^{j+1}_X
\lo {\cal W}\Om^{j+1}_X/V\lo 0, 
\end{align*}  
we have the following injection 
\begin{align*} 
H^q(X,{\cal W}\Om^{j+1}_X)/V\os{\subset}{\lo} 
H^q(X,{\cal W}\Om^{j+1}_X/V). 
\tag{5.1.6}\label{ali:wooivx} 
\end{align*} 
Since the following diagram 
\begin{equation*} 
\begin{CD}
H^q(X,{\cal W}\Om^j_X)/F
@>{\sim}>> H^q(X,{\cal W}\Om^j_X/F)
\\
@V{dV}VV @VV{dV}V \\
H^q(X,{\cal W}\Om^{j+1}_X)/V
@>>> H^q(X,{\cal W}\Om^{j+1}_X/V) 
\end{CD} 
\end{equation*}
is commutative, 
the morphism 
$$dV \col H^{q}(X,{\cal W}\Om^j_X/F)
\lo 
H^{q}(X,{\cal W}\Om^{j+1}_X/V)$$
is zero. 
Hence we see that the morphism 
$$H^{q-1}(X,\Om^{i+1}_{X/s})\lo H^{q}(X,{\cal W}\Om^{i}_X/F)
=H^q(X,{\cal W}\Om^{i}_X)/F$$ 
is surjective by considering the case $j=i$ in (\ref{ali:wmnii}).  
We also see that the morphism 
$$
H^{q}(X,{\cal W}\Om^{i}_X/V)
\lo 
H^{q}(X,\Om^{i}_{X/s})$$ 
is injective by considering the case $j=i-1$ in (\ref{ali:wmnii}). 
We have proved (\ref{theo:htb}). 
\end{proof}

\begin{coro}\label{coro:hb}
Let the assumptions be as in {\rm (\ref{theo:htb})}. 
Let $G^{iq}$ be the $p$-divisible group whose Cartier module 
is $H^q(X,{\cal W}\Om^{i}_X)$. 
Let $(G^{iq})^*$ be the Cartier dual of $G^{iq}$. 
Let $h(G^{iq})$ be the height of $G^{iq}$. 
Then 
\begin{align*} 
\dim G^{iq}\leq \dim_{\kap}
H^{q}(X,\Om^{i}_{X/s})
\tag{5.2.1}\label{ali:fxsi}
\end{align*}  
\begin{align*} 
\dim (G^{iq})^*\leq \dim_{\kap}
H^{q-1}(X,\Om^{i+1}_{X/s}).
\tag{5.2.2}\label{ali:fxdsi}
\end{align*}  
and 
\begin{align*} 
h(G^{iq})\leq {\dim}_{\kap}
H^{q-1}(X,\Om^{i+1}_{X/s})+\dim_{\kap}
H^{q}(X,\Om^{i}_{X/s}).
\tag{5.2.3}\label{ali:fki}
\end{align*}  
\end{coro} 
\begin{proof} 
Because $\dim G^{iq}= \dim_{\kap}(H^q(X,{\cal W}\Om^{i}_X)/V)$, 
$\dim (G^{iq})^*= \dim_{\kap}(H^q(X,{\cal W}\Om^{i}_X)/F)$, 
and $h(G^{iq})=\dim_{\kap}H^q(X,{\cal W}\Om^{i}_X)/p$, 
we obtain the inequality (\ref{ali:fxsi}), (\ref{ali:fxdsi}) 
and 
(\ref{ali:fki}), respectively, by (\ref{eqn:xxs}). 
\end{proof}

The following is a generalization of \cite[(2.3)]{vgkht}; 
our assumption is much weaker than that of [loc.~cit.]:

\begin{coro}\label{coro:amh}
{\rm (\ref{theo:amh})} holds. 
\end{coro} 
\begin{proof} 
By the assumption, $H^q(X,{\cal W}({\cal O}_X))$ is a free ${\cal W}$-module 
of finite rank $h^q(\os{\circ}{X}/\kap)$.  
The induced morphism $d\col H^q(X,{\cal W}({\cal O}_X))\lo 
H^q(X,{\cal W}\Om^1_{X})$ by the derivative 
$d\col {\cal W}({\cal O}_X)\lo {\cal W}\Om^1_{X}$ 
is zero by \cite[(2.5)]{nyc} or the log version of \cite[II (3.8)]{ir}. 
Now (\ref{theo:amh}) follows from (\ref{coro:hb}). 
\end{proof} 

\begin{exem} 
Let $X/s$ be a log Calabi-Yau variety of pure dimension $d$. 
Assume that $h^d(\os{\circ}{X}/\kap)$ is finite. 
Then $h^d(\os{\circ}{X}/\kap)\leq H^{d-1}(X,\Om^{1}_{X/s})+1$. 
This is a log version of \cite[(2.4)]{vgkht}. 
As in the trivial logarithmic case, we say that 
$X/s$ is {\it rigid} if $H^{d-1}(X,\Om^{1}_{X/s})=0$. 
(This is equivalent to the vanishing of 
$H^1(X,\Om^{d-1}_{X/s})$ by the log Serre duality of Tsuji 
(\cite[(2.21)]{tsp}).)
Consequently the height of a rigid log Calabi-Yau variety is $1$ or $\infty$. 
\end{exem}

\begin{theo}\label{theo:dl}
Let $q$ and $i$ be nonnegative integers. 
Let the assumptions be as in {\rm (\ref{theo:htb})}. 
However, instead of the injectivity of the morphism 
$$F\col H^q(X,{\cal W}\Om^{i}_X)\lo H^q(X,{\cal W}\Om^{i}_X)$$ 
in {\rm (\ref{theo:htb})}, 
assume that the operator 
$$V\col H^q(X,{\cal W}\Om^{i}_X)\lo H^q(X,{\cal W}\Om^{i}_X)$$  
is injective. 
Then there exists the following diagram 
\begin{equation*} 
\begin{CD} 
0@>>>H^q(X,{\cal W}\Om^{i}_X)/F@>{V}>> H^q(X,{\cal W}\Om^{i}_X)/p@>>>
H^q(X,{\cal W}\Om^{i}_X)/V@>>> 0\\
@. @V{\simeq}VV@. @VV{\bigcap}V\\
@. H^q(X,{\cal W}\Om^{i}_X/F)@. @.H^q(X,{\cal W}\Om^{i}_X/V)\\
@. @AAA@. @VV{\bigcap}V\\
@. H^{q-1}(X,\Om^{i+1}_{X/s})@. @. H^{q}(X,\Om^{i}_{X/s}). 
\end{CD} 
\tag{5.5.1}\label{eqn:xxys}
\end{equation*}  
\end{theo}
\begin{proof} 
By the new assumption, 
we have the following exact sequence 
\begin{align*} 
0\lo H^q(X,{\cal W}\Om^{i}_X)/F
\os{V}{\lo} H^q(X,{\cal W}\Om^{i}_X)/p
\lo H^q(X,{\cal W}\Om^{i}_X)/V\lo 0. 
\tag{5.5.2}\label{ali:wfvnnii}
\end{align*} 
The rest of the proof is the same as that of (\ref{theo:htb}). 
\end{proof}

\parno 
The following is a log version of \cite[II (4.6)]{idw}. 

\begin{coro}[{\bf cf.~\cite[II (4.6)]{idw}}]\label{coro:lvi}
Let $q$ be a nonnegative integer. 
For any $i$ and $j$ such that $i+j=q$, 
assume that the operators 
$$V\col H^j(X,{\cal W}\Om^{i}_X)\lo H^j(X,{\cal W}\Om^{i}_X),$$ 
$$F\col H^{j+1}(X,{\cal W}\Om^i_X)\lo H^{j+1}(X,{\cal W}\Om^i_X)$$
are injective and that the operator 
$$dV \col H^j(X,{\cal W}\Om^i_X)\lo H^q(X,{\cal W}\Om^{j+1}_X)$$ 
is zero. 
Then there exists the following exact sequence 
\begin{align*} 
0 &\lo  H^0(X,{\cal W}\Om^q_X)/p
\lo H^0(X,\Om^q_{X/s}) \lo H^1(X,{\cal W}\Om^{q-1}_X)/p
\lo H^1(X,\Om^{q-1}_{X/s})\tag{5.6.1}\label{ali:il}\\
&\lo H^2(X,{\cal W}\Om^{q-2}_X)/p\lo H^2(X,\Om^{q-2}_{X/s})
\lo \cdots \lo \cdots\\
&\lo H^{q-1}(X,\Om^{1}_{X/s})\lo 
H^q(X,{\cal W}({\cal O}_X))/p\lo H^q(X,{\cal O}_X)\lo 0. 
\end{align*} 
\end{coro} 
\begin{proof} 
(\ref{coro:lvi}) follows from (\ref{theo:dl}). 
\end{proof}

\begin{rema}\label{rema:fsls} 
(1) We leave the log version of \cite[II (4.5)]{idw} to the reader. 
\par 
(2) In the trivial log case, the assumptions in (\ref{coro:lvi}) 
are slightly weaker than those in \cite[II (4.6)]{idw}. 
\end{rema}

\section{Ordinary log schemes and $F$-split log schemes}\label{sec:fs} 
In this section we give the definition of the ordinarity  at a bidegree 
for a proper log smooth scheme of Cartier type.  
We also generalize results in \cite{j} and \cite{jr} for $F$-split log schemes over $s$. 
We also prove that the nontrivial exotic torsions of 
log crystalline cohomologies of 
$F$-split proper log smooth schemes do not exist. 
This is a log version of Joshi's 
result (\cite{j}). 
We also give the criterion of the $F$-splitness for certain log schemes. 
\par
Let the notations be as in the previous section. 

\begin{defi}[{\bf cf.~\cite[(7.2)]{blk}, \cite[IV (4.12), (4.13)]{ir}, 
\cite[\S4]{lodw}}]\label{defi:wo}
Let $q$ be a nonnegative integer. 
\par 
(1) We say that $X/s$ is {\it ordinary at} $(0,q)$ if $H^q(X,B\Om^1_{X/s})=0$. 
\par 
(2) We say that $X/s$ is {\it ordinary at} $(0,\star)$ (or simply at $0$)
if $H^q(X,B\Om^1_{X/s})=0$ for any $q\in {\mab N}$. 
\end{defi}

\begin{prop}\label{prop:wo}
Let $q$ be a nonnegative integer. 
Then the following are equivalent$:$
\par 
$(1)$ $X/s$ is ordinary at $(0,q)$. 
\par 
$(2)$
For any $n\in {\mab Z}_{\geq 1}$, $H^q(X,B_n\Om^1_{X/s})=0$. 
\par 
$(3)$ For any $n\in {\mab Z}_{\geq 1}$, 
$H^q(X,{\cal W}_n({\cal O}_X))/F=0={}_FH^{q+1}(X,{\cal W}_n({\cal O}_X))$.
\par 
$(4)$ $H^q(X,{\cal O}_X)/F=0={}_FH^{q+1}(X,{\cal O}_X)$. 
\end{prop}
\begin{proof} 
$(1)\Lo (2)$: 
Recall the right vertical exact sequence in (\ref{cd:dnbx}):  
\begin{align*} 
0\lo F_*(B_{n-1}\Om^1_{X/s})\lo B_n\Om^1_{X/s}\os{C^{n-1}}{\lo}B_1\Om^1_{X/s}\lo 0. 
\tag{6.2.1}\label{ali:fbb} 
\end{align*} 
By noting that 
$H^q(X,F_*(B_{n-1}\Om^1_{X/s}))\simeq H^q(X,B_{n-1}\Om^1_{X/s})$ 
and that $H^q(X,B_{1}\Om^1_{X/s})\simeq H^q(X,B\Om^1_{X/s})$ 
and using induction on $n$,  
we obtain the implication $(1)\Lo (2)$. 
\par 
$(2)\Lo (3)$: 
By (\ref{ali:wnox}) we have the following exact sequence: 
\begin{align*} 
0\lo H^q(X,{\cal W}_n({\cal O}_X))/F\lo H^q(X,B_n\Om^1_{X/s})\lo 
{}_FH^{q+1}(X,{\cal W}_n({\cal O}_X))\lo 0. 
\tag{6.2.2}\label{ali:fex}
\end{align*} 
Hence we obtain the implication $(2)\Lo (3)$. 
\par 
$(3)\Lo (4)$: This is obvious. 
\par 
$(4)\Lo (1)$: By (\ref{ali:fex}) for the case $n=1$, 
we have the following exact sequence: 
\begin{align*} 
0\lo H^q(X,{\cal O}_X)/F\lo H^q(X,B_1\Om^1_{X/s})\lo 
{}_FH^{q+1}(X,{\cal O}_X)\lo 0. 
\tag{6.2.3}\label{ali:qbq}
\end{align*} 
Hence we obtain the implication $(4)\Lo (1)$. 
\end{proof} 

\begin{rema}\label{rema:od}
(1) (\cite[(1.4)]{blk}, (resp.~\cite[(1.3)]{nyt})) 
Let $X/s$ be an abelian variety 
(resp.~$K3$-surface) over $\kap$. 
Then $X/s$ is ordinary if and only if 
it is ordinary at $(0,1)$ (resp.~$(0,2)$). 
\par 
(2) The ordinarity at $(0,\star)$ is an interesting notion: 
see \cite{st} for the main theorem in [loc.~cit.].  
\end{rema}

\par 
We need the following remark for (\ref{prop:woni}) below. 

\begin{rema}\label{rema:rld}
(1) Let the notations be as in \cite[p.~256]{lodw}. 
Let ${\cal J}_n$ be the ideal sheaf of the closed immersion $X\os{\sus}{\lo} Z_n$. 
The definition of the morphism 
\begin{align*} 
d\log \col {\cal M}_X^{\rm gp}\vert_U\lo W_n\Om^1_{U/S}
\tag{6.4.1}\label{ali:mxwnus}
\end{align*} 
is mistaken in [loc.~cit.]. Though the lift $\wt{m}$ of a local section $m$ of 
${\cal M}_X^{\rm gp}\vert_U$ is taken in ${\cal M}_{D_n}^{\rm gp}$ 
in [loc.~cit.], we have to take a lift $\wt{m}$ in ${\cal M}^{\rm gp}_{Z_n}$
because 
$W_n\Om^1_{U/S}=
{\cal H}^1({\cal O}_{D_n}\otimes_{{\cal O}_{Z_n}}\Om^{\bul}_{Z_n/S_n})$. 
Furthermore, we have to take the cohomology class'' of $1\otimes d\log \wt{m}$ 
to define the morphism (\ref{ali:mxwnus}). 
If one uses an isomorphism 
\begin{equation*}
\Om^{\bul}_{D_n/S_n, [~]}
\os{\sim}{\lo} 
{\cal O}_{D_n}
\otimes_{{\cal O}_{Z_n}}
\Om^{\bul}_{Z_n/S_n}.   
\tag{6.4.2}\label{eqn:fwniwu}
\end{equation*}   
proved in \cite[(1.3.28.1)]{nb}, one can define the image of 
$m$ by the morphism (\ref{ali:mxwnus}) 
as the cohomology class of the image $d\log \wt{m}$, 
where $\wt{m}$ is a lift of $m$ in ${\cal M}_{D_n}^{\rm gp}$. 
Here $\Om^{\bul}_{D_n/S_n, [~]}$ is the quotient of 
$\bigoplus_{i\in {\mab N}}\Om^i_{D_n/S_n}$ by the ideal sheaf generated 
by local sections of the form $d(a^{[e]})-a^{[e-1]}da$ 
$(a\in {\cal J}_n, e\in {\mab Z}_{\geq 1})$.
\par 
(2) For a positive integer $q$, the definition of $W_n\Om^q_{X/S,{\log}}$ 
is not perfect in \cite[p.~257]{lodw}. 
The right definition of $W_n\Om^q_{X/S,{\log}}$ is as follows. 
The sheaf $W_n\Om^q_{X/S,{\log}}$ 
is an abelian subsheaf of $W_n\Om^q_{X/S}$ generated by 
the image of the following composite morphism 
\begin{align*} 
({\cal M}_X^{\rm gp})^{\otimes q}\os{(d\log)^{\otimes q}}{\lo} 
(W_n\Om^1_{X/S,\log})^{\otimes q}\lo (W_n\Om^1_{X/S})^{\otimes q}
\lo W_n\Om^q_{X/S}. 
\end{align*} 
Here all the tensor products are taken over ${\mab Z}$ and 
the morphism $(W_n\Om^1_{X/S})^{\otimes q}
\lo W_n\Om^q_{X/S}$ is the following local wedge product: 
\begin{align*} 
\underset{q~{\rm times}}
{\underbrace{[?]\wedge [?]\wedge \cdots [?] \wedge [?]}}: 
({\cal H}^1({\cal O}_{D_n}\otimes_{{\cal O}_{Z_n}}\Om^{\bul}_{Z_n/S_n}))^{\otimes q}
\lo 
{\cal H}^q({\cal O}_{D_n}\otimes_{{\cal O}_{Z_n}}\Om^{\bul}_{Z_n/S_n}). 
\end{align*} 
Furthermore, set $W_n\Om^0_{X/S,{\log}}:={\mab Z}/p^n$ on $\os{\circ}{X}_{\rm et}$.  
It is a routine work to check that this local wedge product is independent of the choice 
of the immersion $X\os{\sus}{\lo}Z_n$.   
\end{rema}

\par
Though all the following statements are not included in \cite[(7.3)]{blk},  
\cite[IV (4.13)]{ir} and \cite[(4.1)]{lodw}, 
almost all of them are essentially included in [loc.~cit.].  

\begin{prop}[{\bf cf.~\cite[(7.3)]{blk}, \cite[IV (4.13)]{ir},
\cite[(4.1)]{lodw}}]\label{prop:woni}
Let $q$ be a nonnegative integer. 
Denote ${\cal W}_1\Om^i_{X,\log}$ by $\Om^i_{X/s,\log}$ by abuse of notation. 
Then the following are equivalent$:$
\par 
$(1)$ $X/s$ is ordinary at $(0,\star)$. 
\par 
$(2)$
For any $n\in {\mab Z}_{\geq 1}$ and for any $q\in {\mab N}$, 
$H^q(X,B_n\Om^1_{X/s})=0$. 
\par 
$(3)$ 
For a positive integer $n$ and for any $q\in {\mab N}$, 
$H^q(X,B_n\Om^1_{X/s})=0$. 
\par 
$(4)$ 
For any $n\in {\mab Z}_{\geq 1}$ and any $q\in {\mab N}$, the operator 
\begin{align*} 
F\col H^q(X,{\cal W}_n({\cal O}_X))\lo H^q(X,{\cal W}_n({\cal O}_X))
\end{align*} 
is bijective. 
\par 
$(5)$ For a positive integer $n$ and any $q\in {\mab N}$, the operator 
\begin{align*} 
F\col H^q(X,{\cal W}_n({\cal O}_X))\lo H^q(X,{\cal W}_n({\cal O}_X))
\end{align*} 
is bijective. 
\par 
$(6)$ For any $q\in {\mab N}$, the operator 
\begin{align*} 
F\col H^q(X,{\cal O}_X)\lo H^q(X,{\cal O}_X)
\end{align*} 
is bijective. 
\par 
$(7)$ For any $q\in {\mab N}$, the operator 
\begin{align*} 
F\col H^q(X,{\cal W}({\cal O}_X))\lo H^q(X,{\cal W}({\cal O}_X))
\end{align*} 
is bijective. 
\par 
$(8)$ 
Set $X_{\ol{\kap}}:=X\otimes_{\kap}\ol{\kap}$. 
Then $\dim_{{\mab F}_p}H^q_{\rm et}(X_{\ol{\kap}},{\mab F}_p)
=\dim_{\kap}H^q(X,{\cal O}_{X})$ for any $q\in {\mab N}$. 
\par  
$(9)$  For any $q\in {\mab N}$, 
the natural morphism 
\begin{align*} 
H^q_{\rm et}(X_{\ol{\kap}},{\mab F}_p)\otimes_{{\mab F}_p}\ol{\kap} 
\lo H^q(X_{\ol{\kap}},{\cal O}_{X_{\ol{\kap}}})
\end{align*} 
is an isomorphism. 
\par 
$(10)$ For any $n\in {\mab Z}_{\geq 1}$ and for any $q\in {\mab N}$, 
the natural morphism 
\begin{align*} 
H^q_{\rm et}(X_{\ol{\kap}},{\mab Z}/p^n)
\otimes_{{\mab Z}/p^n}{\cal W}_n(\ol{\kap})
\lo H^q(X_{\ol{\kap}},{\cal W}_n({\cal O}_{X_{\ol{\kap}}}))
\end{align*} 
is an isomorphism. 
\par 
$(11)$ 
For any $q\in {\mab N}$, $H^q(X,B{\cal W}\Om^1_{X})=0$.
\par 
$(12)$ 
For any $n\in {\mab Z}_{\geq 1}$ and for any $q\in {\mab N}$, 
$H^q(X,B{\cal W}_n\Om^1_{X})=0$.
\end{prop}
\par 
$(13)$ The natural morphism 
$H^q(X,\Om^1_{X/s,\log})\otimes_{{\mab F}_p}\kap \lo 
H^q(X,\Om^1_{X/s})$ is an isomorphism for any $q\in {\mab N}$. 
\par 
$(14)$ The natural morphism 
$H^q(X,{\cal W}_n\Om^1_{X,\log})\otimes_{{\mab F}_p}\kap \lo 
H^q(X,{\cal W}_n\Om^1_{X})$ is an isomorphism for any $n\in {\mab Z}_{\geq 1}$ 
and any $q\in {\mab N}$. 
\par 
$(15)$ The natural morphism 
$H^q(X_{\ol{\kap}},
{\cal W}\Om^1_{X_{\ol{\kap}},\log})\otimes_{{\mab Z}_p}{\cal W}(\ol{\kap}) \lo 
H^q(X_{\ol{\kap}},{\cal W}\Om^1_{X_{\ol{\kap}}})$ 
is an isomorphism for any $q\in {\mab N}$. 
\begin{proof} 
The implications $(1)\Lo (2)\Lo (3)$,  $(2) \iff (4)$, $(3)\iff (5)$, $(1) \iff (6)$, $(6) \Lo (7)$ 
$(12) \Lo (1)$, $(12)\Lo (11)$, $(14)\Lo (13)$  and $(14)\Lo (15)$ 
immediately follows from (\ref{prop:wo}) or obvious. 
Hence it suffices to prove the implications $(3)\Lo (1)$, $(6)\iff (8)\iff (9)$, 
$(9)\iff (10)$, $(1)\iff (13)\Lo (14)$,  $(15)\Lo (7)\Lo (12)$  and $(11)\Lo (1)$. 
Assume that (3) holds. 
Let $n$ be a positive integer in (3).  
By (\ref{ali:fbb}) we have 
the following exact sequence of abelian groups: 
\begin{align*} 
\lo H^q(X,B_{n-1}\Om^1_{X/s})\lo H^q(X,B_n\Om^1_{X/s})
\lo H^q(X,B_1\Om^1_{X/s})\lo H^{q+1}(X,B_{n-1}\Om^1_{X/s}). 
\tag{6.5.1}\label{ali:qbbq}
\end{align*} 
Hence 
\begin{align*} 
H^q(X,B_1\Om^1_{X/s})=H^{q+1}(X,B_{n-1}\Om^1_{X/s})\quad (\forall q).
\tag{6.5.2}\label{ali:qexqxq}
\end{align*} 
\par 
On the other hand,  we have the following exact sequence 
of abelian sheaves by the definition of $B_n\Om^1_{X/s}$: 
\begin{align*} 
0\lo B_1\Om^1_{X/s}\lo B_n\Om^1_{X/s}
\lo B_{n-1}\Om^1_{X'/s}\lo 0. 
\tag{6.5.3}\label{ali:qexq}
\end{align*} 
Taking the long exact sequence of (\ref{ali:qexq}) and using  
the assumption, we have the following equality 
\begin{align*} 
H^q(X,B_{n-1}\Om^1_{X'/s})=H^{q+1}(X,B_1\Om^1_{X/s})\quad (\forall q).
\tag{6.5.4}\label{ali:qexxq}
\end{align*} 
Here we have identified abelian sheaves on 
$\os{\circ}{X}$ with those on $\os{\circ}{X}{}'$. 
By \cite[(1.13)]{lodw} 
the sheaf $B_m\Om^1_{X/s}$ $(m\in {\mab N})$ is a locally free sheaf of 
${\cal O}_X$-modules of finite rank and 
it commutes with the base changes of $s$. 
Hence $B_m\Om^1_{X'/s}=
\kap \otimes_{\sig,\kap}B_m\Om^1_{X/s}\simeq B_m\Om^1_{X/s}$, 
where $\sig$ is the Frobenius automorphism of $\kap$. 
Hence we have the following equality by (\ref{ali:qexxq}): 
\begin{align*} 
H^q(X,B_{n-1}\Om^1_{X/s})=H^{q+1}(X,B_1\Om^1_{X/s})\quad (\forall q).
\tag{6.5.5}\label{ali:qepxq}
\end{align*} 
By (\ref{ali:qexqxq}) and (\ref{ali:qepxq}) 
we have the following equality: 
\begin{align*} 
H^q(X,B_1\Om^1_{X/s})=H^{q+2}(X,B_1\Om^1_{X/s})\quad (\forall q).
\tag{6.5.6}\label{ali:qebpxq}
\end{align*}
If $q> \dim \os{\circ}{X}$, $H^q(X,B_1\Om^1_{X/s})=0$. 
Hence $H^q(X,B_1\Om^1_{X/s})=0$ for any $q\in {\mab N}$. 
We have proved the implication $(3)\Lo (1)$. 
\par 
By the following exact sequence 
\begin{align*} 
0\lo {\mab F}_p\lo {\mab G}_a \os{1-F}{\lo} {\mab G}_a\lo 0
\end{align*} 
on $(X_{\ol{\kap}})_{\rm et}$ and using the surjectivity of 
the morphism $1-F \col H^q(X_{\ol{\kap}},{\cal O}_{X_{\ol{\kap}}})\lo 
H^q(X_{\ol{\kap}},{\cal O}_{X_{\ol{\kap}}})$ (\cite[II (5.3)]{idw}),  
we have the following exact sequence 
\begin{align*} 
0\lo H^q_{\rm et}(X_{\ol{\kap}},{\mab F}_p)\lo 
H^q(X_{\ol{\kap}},{\cal O}_{X_{\ol{\kap}}}) 
\os{1-F}{\lo} H^q(X_{\ol{\kap}},{\cal O}_{X_{\ol{\kap}}})\lo 0. 
\end{align*} 
The implications $(6)\iff (8)\iff (9)$ are special cases of \cite[(3.3)]{cl} 
and \cite[\S2]{mus}. 
\par 
Since $F\col {\cal W}_{n}({\mab F}_p)\lo {\cal W}_{n}({\mab F}_p)$ 
is the identity of ${\cal W}_{n}({\mab F}_p)$ and $FV=p$, 
we have the following commutative diagram: 
\begin{equation*} 
\begin{CD} 
0@>>>{\mab Z}/p^{n-1}@>{p}>> {\mab Z}/p^{n} @>>>{\mab Z}/p@>>> 0\\
@. @| @| @| @.\\
0@>>>{\cal W}_{n-1}({\mab F}_p)
@>{V}>> {\cal W}_{n}({\mab F}_p)
@>>>{\mab F}_p@>>> 0\\
@. @VVV @VVV @VVV @. \\
0@>>>{\cal W}_{n-1}({\cal O}_{X_{\ol{\kap}}})
@>{V}>> {\cal W}_n({\cal O}_{X_{\ol{\kap}}}) 
@>>>{\cal O}_{X_{\ol{\kap}}}@>>> 0. 
\end{CD} 
\tag{6.5.7}\label{cd:obc}
\end{equation*} 
Since the morphism ${\mab Z}/p^n\lo {\cal W}_n(\ol{\kap})$ 
is flat and since 
$H^q_{\rm et}(X_{\ol{\kap}},{\mab Z}/p^{n-1})
\otimes_{{\mab Z}/p^n}{\cal W}_n(\ol{\kap})
=H^q_{\rm et}(X_{\ol{\kap}},{\mab Z}/p^{n-1})
\otimes_{{\mab Z}/p^{n-1}}{\cal W}_{n-1}(\ol{\kap})$, 
we have the following exact sequence: 
\begin{align*} 
\cdots &\lo H^q_{\rm et}(X_{\ol{\kap}},{\mab Z}/p^{n-1})
\otimes_{{\mab Z}/p^{n-1}}{\cal W}_{n-1}(\ol{\kap})
\lo H^q_{\rm et}(X_{\ol{\kap}},{\mab Z}/p^n)
\otimes_{{\mab Z}/p^n}{\cal W}_n(\ol{\kap})\tag{6.5.8}\label{ali:obc}\\
&\lo H^q_{\rm et}(X_{\ol{\kap}},{\mab F}_p)
\otimes_{{\mab F}_p}\ol{\kap} \lo \cdots.  
\end{align*} 
Now the implication $(9)\Lo (10)$ follows from 
the commutative diagram (\ref{cd:obc}), 
the exact sequence (\ref{ali:obc})  
and the induction on $n$: 
The implication $(10)\Lo (9)$ is obvious. 
\par 
The implication $(7)\Lo (12)$ follows from a special case of the following 
exact sequence proved in \cite[p.~263]{lodw}: 
\begin{align*} 
0&\lo  H^q(X,{\cal W}{\Om}^{i-1}_X)/(F^n+V^n)
H^q(X, {\cal W}{\Om}^{i-1}_X)\os{d}{\lo}
H^q(X,B{\cal W}_n{\Om}^i_X)\\
&\lo (V^n)^{-1}F^n 
H^{q+1}(X, {\cal W}{\Om}^{i-1}_X)/
F^nH^{q+1}(X, {\cal W}{\Om}^{i-1}_X)\lo 0 \quad (i\in {\mab N}). 
\end{align*} 
\par 
The equivalence $(7)\iff (12)$ has been essentially noted in \cite[(4.2)]{lodw}. 
\par 
The equivalence $(1)\iff (13)$ follows from the following exact sequence 
\begin{align*} 
0\lo \Om^i_{X/s,\log}\lo \Om^i_{X/s}\os{1-C^{-1}}{\lo} \Om^i_{X/s}/B\Om^i_{X/s}\lo 0
\end{align*} 
(\cite[p.~262]{lodw}, (cf.~\cite[(6.1.1)]{tsy}, \cite[(4.1)]{nlk3})). 
\par 
The implication $(13)\Lo (14)$ follows from the following exact sequence 
\begin{align*} 
0\lo {\cal W}_m\Om^i_{X,{\log}}\os{{\bf p}^n}{\lo} 
{\cal W}_{m+n}\Om^i_{X,{\log}} \lo {\cal W}_n\Om^i_{X,{\log}}\lo 0,
\end{align*} 
which has been proved in \cite[(2.12)]{lodw}. 
\par 
To prove the implication $(15)\Lo (7)$, we may assume that 
$\kap=\ol{\kap}$. In this case, the implication follows from 
the equality 
$H^q(X,{\cal W}\Om^i_{X,{\log}})=H^q(X,{\cal W}\Om^i_X)^F$ 
(\cite[(3.4.1)]{lodw}). 
\par 
The implication $(11)\Lo (1)$ has been essentially noted in \cite[(4.2)]{lodw}. 
\par 
We have completed the proof of (\ref{prop:woni}). 
\end{proof}

\begin{rema}\label{rema:woe}
I do not know whether the statement with 
the replacement of ``any $n\in {\mab Z}_{\geq 1}$'' in (10) 
by ``a positive integer $n$'' is equivalent to (1). 
\end{rema}

\begin{rema}\label{rema:goc}
As in \cite[Conjecture 1.1]{mus}, one can conjecture the following: 
\par
Let $X$ be a proper smooth scheme over a field of characteristic zero. 
Let ${\cal X}$ be a proper flat model of $X$ 
over a ${\mab Z}$-algebra $A$ of finite type. 
Then there exists a dense set of closed points $T\subset {\rm Spec}(A)$ 
such that ${\cal X}_t/t$ is ordinary at $(0,\star)$ for every $t\in T$. 
\end{rema}

\begin{prop}\label{prop:st}
The following hold$:$
\par 
$(1)$ Let $q$ be a nonnegative integer. 
Assume that $X/s$ is ordinary at $(0,q-1)$. 
Then the submodule of $p$-torsions of $H^q(X,{\cal W}({\cal O}_X))$ 
is equal to that of $V$-torsions of 
$H^q(X,{\cal W}({\cal O}_X))$.  
\par 
$(2)$ Assume that $\os{\circ}{X}$ is $F$-split. 
Then $X/s$ is ordinary at $(0,\star)$. 
\end{prop}
\begin{proof} 
(1): 
Since ${}_FH^q(X,{\cal W}({\cal O}_X))=\vpl_n{}_FH^q(X,{\cal W}_n({\cal O}_X))$, 
${}_FH^q(X,{\cal W}({\cal O}_X))=0$ by (\ref{prop:wo}). 
Since $FV=p$, we immediately obtain (1). 
\par 
(2): The proof of (2) is the same as that of \cite[(2.4.1)]{jr} 
by using the log Serre's exact sequence 
(\ref{ali:wnox}) for the case $n=1$. 
\end{proof}

\par 
Let $q$ be a nonnegative integer. 
Let $H^q_{\rm crys}(X/{\cal W}(s))$ be the log crystalline cohomology of 
$X/{\cal W}(s)$ (\cite{klog1}). 
Next we discuss exotic torsions in $H^q_{\rm crys}(X/{\cal W}(s))$ as in \cite{j}. 
\par 
Let $q$ be a nonnegative integer. 
Set $Q^q:={\rm Im}
(H^q_{\rm crys}(X/{\cal W}(s))_{\rm tor}\lo H^q(X,{\cal W}({\cal O}_X)))$. 
As in \cite[II (6.7.3)]{idw},  
we define the module $H^q_{\rm crys}(X/{\cal W}(s))_e$ 
of exotic torsions in $H^q_{\rm crys}(X/{\cal W}(s))$ as the following quotient 
\begin{align*} 
H^q_{\rm crys}(X/{\cal W}(s))_e:=
Q^q/(H^q(X,{\cal W}({\cal O}_X))_{V_{\rm tor}}\cap Q^q). 
\end{align*} 
(In [loc.~cit.] only the case $q=2$ has been considered.)

The following is a log version of a generalization of \cite[(7.3)]{j}. 

\begin{prop}\label{prop:lvof}
Assume that $X/s$ is ordinary at $(0,q)$.  
Then $H^q_{\rm crys}(X/{\cal W}(s))_e=0$. 
\end{prop}
\begin{proof} 
This follows from (\ref{prop:st}) (1). 
\end{proof}


\begin{coro}\label{coro:fspv}
Assume that $\os{\circ}{X}$ is $F$-split.  
Then $H^q_{\rm crys}(X/{\cal W}(s))_e=0$ $(q\in {\mab N})$. 
\end{coro} 
\begin{proof} 
This follows from (\ref{prop:st}) (2) and (\ref{prop:lvof}). 
\end{proof}

The following is a log version of \cite[(2.4.2)]{jr} with slightly weaker assumption. 
Our proof is slightly more immediate than the proof in [loc.~cit.].

\begin{prop}\label{prop:exfsp}
Assume that, $X/s$ is of vertical type, and 
that $\Om^d_{X/s}$ is trivial and 
that $X/s$ is ordinary at $(0,d-1)$. 
Then $\os{\circ}{X}$ is $F$-split.
\end{prop}
\begin{proof} 
By using the log Serre duality of Tsuji 
(\cite[(2.21)]{tsp}) and using the ordinarity at $(0,d-1)$, 
we have the following equalities: 
\begin{align*} 
{\rm Ext}_X^1(B_1\Om^1_{X/s},{\cal O}_X)=
{\rm Ext}_X^1(B_1\Om^1_{X/s},\Om^d_{X/s})
=H^{d-1}(X,B_1\Om^1_{X/s})^*=0. 
\end{align*} 
Here $*$ means the dual vector space. 
\end{proof}



\par 
If $\dim \os{\circ}{X}\leq 2$, 
we can give explicit examples easily for 
an $F$-split proper degenerate log variety  
by the classification of lower dimensional proper smooth varieties. 

\begin{prop}\label{prop:nel}
Assume that $s$ is the log point of ${\rm Spec}(\kap)$. 
Let $X$ be a proper log Calabi-Yau variety over $s$. 
Assume that $\os{\circ}{X}/\kap$ is not smooth. 
Then the following hold$:$ 
\par 
$(1)$ Assume that $\os{\circ}{X}$ is of pure dimension $1$. 
$($In this case we say that $X/s$ is a {\it log elliptic curve}.$)$ 
Then $\os{\circ}{X}$ is $F$-split. 
\par 
$(2)$ Assume that $\os{\circ}{X}$ is of pure dimension $2$.
$($In {\rm \cite{nlk3}}, in this case, 
we have said that $X/s$ is a log $K3$-surface.$)$ 
If $\os{\circ}{X}$ is of Type II {\rm (\cite[\S3]{nlk3})}, 
then $\os{\circ}{X}$ is $F$-split if and only if 
the isomorphic double elliptic curve of $\os{\circ}{X}$ is ordinary. 
If $\os{\circ}{X}$ is of Type III {\rm ([loc.~cit.])}, 
then $\os{\circ}{X}$ is $F$-split. 
\par
$(3)$ Let the notations be as in $(2)$. 
If $\os{\circ}{X}$ is of Type II {\rm (\cite[\S3]{nlk3})} and if 
the isomorphic double elliptic curve of $\os{\circ}{X}$ is supersingular, 
then $h_F(\os{\circ}{X})=2$.  
\end{prop} 
\begin{proof} 
Set $d:=\dim \os{\circ}{X}$. 
By (\ref{theo:cyh}), $\os{\circ}{X}$ is $F$-split if and only if 
$h^d(\os{\circ}{X}/\kap)=1$. 
Let $\os{\circ}{X}{}^{(i)}$ $(i\in {\mab Z}_{\geq 0})$ 
be the disjoint union of the $(i+1)$-fold 
intersections of the different irreducible components of 
$\os{\circ}{X}$.  
Then, by \cite[Theorem 1]{rs}, we have the following spectral sequence 
\begin{align*} 
E^{ij}=H^j(\os{\circ}{X}{}^{(i)},{\cal W}({\cal O}_{\os{\circ}{X}{}^{(i)}}))
\Lo H^{i+j}(X,{\cal W}({\cal O}_X))
\tag{6.12.1}\label{ali:xag}
\end{align*}
obtained by the following exact sequence 
\begin{align*} 
0\lo {\cal W}({\cal O}_X)\lo {\cal W}({\cal O}_{\os{\circ}{X}{}^{(0)}})\lo 
{\cal W}({\cal O}_{\os{\circ}{X}{}^{(1)}})\lo \cdots.  
\tag{6.12.2}\label{ali:csfp}
\end{align*}
Let $D(\Phi^q_{\os{\circ}{X}/\kap})$ $(q\in {\mab N}_{\geq 1})$ be 
the Dieudonn\'{e} module of $\Phi^q_{\os{\circ}{X}/\kap}$. 
Then $D(\Phi^q_{\os{\circ}{X}/\kap})=H^q(X,{\cal W}({\cal O}_X))$ 
(\cite{am}). 
\par 
(1): By the easier proof than that of \cite[(3.4)]{nlk3}, 
it is easy to see that $\os{\circ}{X}$ is an $n$-gon $(n\geq 2)$. 
By (\ref{ali:xag}) 
we easily see that 
\begin{align*}D(\Phi^1_{\os{\circ}{X}/\kap})
&=H^1(X,{\cal W}({\cal O}_X))
={\rm Coker}(H^0(\os{\circ}{X}{}^{(0)},{\cal W}({\cal O}_{\os{\circ}{X}{}^{(0)}}))\lo 
H^0(\os{\circ}{X}{}^{(1)},{\cal W}({\cal O}_{\os{\circ}{X}{}^{(1)}})))\\
&={\cal W}.
\end{align*}  
Hence $h^1(\os{\circ}{X}/\kap)=1$. 
By (\ref{theo:cyh}) we obtain (1).  
\par 
(2): By the criterion of \cite[(5.4)]{ndw} and (\ref{theo:cyh}), we obtain (2).  
\par 
(3) Let $E$ be the double elliptic curve over $\kap$. 
By (\ref{ali:xag}) 
we easily see that 
$$D(\Phi^2_{\os{\circ}{X}/\kap})=H^1(X,{\cal W}({\cal O}_X))
=H^1(E,{\cal W}({\cal O}_E)).$$ 
Hence $h^2(\os{\circ}{X}/\kap)=2$. 
By (\ref{theo:cyh}) we obtain (3).  
\end{proof}

\begin{rema}\label{rema:hhf}
I do not know whether 
if $Y/\kap$ is only a combinatorial $K3$-surface of Type II or III, 
then the conclusions of (\ref{prop:nel}) hold. 
\end{rema}

\bigskip
\bigskip
\parno
Yukiyoshi Nakkajima 
\parno
Department of Mathematics,
Tokyo Denki University,
5 Asahi-cho Senju Adachi-ku,
Tokyo 120--8551, Japan. 
\parno
{\it E-mail address\/}: 
nakayuki@cck.dendai.ac.jp

\end{document}